\documentclass[a4paper,11pt]{amsart}
\usepackage{amsfonts}
\usepackage{amssymb}
\usepackage[utf8]{inputenc}
\usepackage{amsmath}
\usepackage{pdflscape}
\usepackage{graphicx}
\usepackage[dvipsnames]{xcolor}
\usepackage[colorlinks=true, linkcolor=red, citecolor=blue]{hyperref}
\usepackage[]{epsfig}
\usepackage[]{pstricks}
\usepackage{tikz}

\providecommand{\U}[1]{\protect\rule{.1in}{.1in}}
\newtheorem{theorem}{Theorem}[section]

\newtheorem{proposition}[theorem]{Proposition}
\newtheorem{lemma}[theorem]{Lemma}
\newtheorem{corollary}[theorem]{Corollary}
\newtheorem{remark}{Remark}

\theoremstyle{remark}

\newcommand{\remove}[1]{ }

\newcommand{\R}{{\mathbb R}}
\newcommand{\T}{{\mathbb T}}
\newcommand{\Z}{{\mathbb Z}}
\newcommand{\N}{{\mathbb N}}
\newcommand{\ft}{{\mathcal{F}}}

\newcommand{\supp}{{\mbox{supp}}}

\newcommand{\px}{\partial_x}
\newcommand{\pt}{\partial_t}

\newcommand{\wt}{\widetilde}
\newcommand{\wh}{\widehat}

\newcommand{\vep}{\varepsilon}

\def\MB{\mathbb}

\newcommand{\bra}[1]{\left\langle #1\right\rangle}

\providecommand{\norm}[1]{\left\lVert#1\right\rVert}
\numberwithin{equation}{section}
\makeatletter
\@namedef{subjclassname@2020}{\textup{2020} Mathematics Subject Classification}
\makeatother

\begin{document}
\title[Control and stability for higher-order KdV-type]{On the control issues for higher-order nonlinear dispersive equations on the circle}
\author[Capistrano--Filho]{Roberto de A. Capistrano--Filho}
\address{Departamento de Matem\'atica, Universidade Federal de Pernambuco, S/N Cidade Universit\'aria, 50740-545, Recife (PE), Brazil}
\email{roberto.capistranofilho@ufpe.br}
\author[Kwak]{Chulkwang Kwak}
\address{Department of Mathematics, Ewha Womans University, Seoul 03760, Korea}
\email{ckkwak@ewha.ac.kr}
\author[Vielma Leal]{Francisco J. Vielma Leal}
\address{Universidade Estadual de Campinas, Instituto de Matem\'atica, Estat\'{i}tica e Computa\c{c}\~{a}o Cient\'{i}fica (IMECC), 13083-859, Campinas (SP), Brazil}
\email{fvielmaleal7@gmail.com ; vielma@ime.unicamp.br}

\begin{abstract}
The local and global control results for a general higher-order KdV-type operator posed on the unit circle are presented. Using spectral analysis, we are able to prove local results, that is, the equation is locally controllable and exponentially stable. To extend the local results to the global one we captured the smoothing properties of the Bourgain spaces, the so-called \textit{propagation of singularities}, which are proved with a new perspective. These propagation, together with the \textit{Strichartz estimates}, are the key to extending the local control properties to the global one, precisely, higher-order KdV-type equations are globally controllable and exponentially stabilizable in the Sobolev space $H^{s}(\mathbb{T})$ for any $s \geq 0$. Our results recover previous results in the literature for the KdV and Kawahara equations and extend, for a general higher-order operator of KdV-type, the Strichartz estimates as well as the propagation results, which are the main novelties of this work.
\end{abstract}
\maketitle


\section{Introduction\label{sec1}}

\subsection{Model description}

The full water wave system is too complex to allow to easily derive \linebreak and rigorously from it relevant qualitative information on the dynamics of the waves. Alternatively, under suitable assumption on amplitude, wavelength, wave steepness and so on, the study on asymptotic models for water waves has been extensively investigated to understand the full water wave system, see, for instance, \cite{BCS, BCL, ASL, BLS, Lannes, Saut} and references therein for a rigorous justification of various asymptotic models for surface and internal waves.

\medskip

Particularly, formulating the waves as a free boundary problem of the incompressible, \linebreak irrotational Euler equation in an appropriate non-dimensional form, one has two non-dimensional parameters $\delta := \frac{h}{\lambda}$ and $\varepsilon := \frac{a}{h}$, where the water depth, the wave length and the amplitude of the free surface are parameterized as $h, \lambda$ and $a$, respectively. Moreover, another non-dimensional parameter $\mu$ is called the Bond number, which measures the importance of gravitational forces compared to surface tension forces. The physical condition $\delta \ll 1$ characterizes the waves, which are called long waves or shallow water waves, but there are several long wave approximations according to relations between $\varepsilon$ and $\delta$, specially,

\medskip

\begin{enumerate}
\item Korteweg-de Vries (KdV): $\varepsilon = \delta^2 \ll 1$ and $\mu \neq \frac13$.
\item Kawahara: $\varepsilon = \delta^4 \ll 1$ and $\mu = \frac13 + \nu\varepsilon^{\frac12}$.
\end{enumerate}

\medskip

Under the regime for $\varepsilon, \delta, \mu$ given in Item (1), Korteweg and de Vries \cite{Korteweg}\footnote{This equation indeed firstly introduced by Boussinesq \cite{Boussinesq}, and Korteweg and de Vries rediscovered it twenty years later.} derived the following equation well-known as a central equation among other dispersive or shallow water wave models called the KdV equation from the equations for capillary-gravity waves: 
\[\pm2 u_t + 3uu_x +\left( \frac13 - \mu\right)u_{xxx} = 0.\]
In connection with the  critical Bond number $\mu = \frac13$, Hasimoto \cite{Hasimoto1970} derived a fifth-order KdV equation of the form
\[\pm2 u_t + 3uu_x - \nu u_{xxx} +\frac{1}{45}u_{xxxxx} = 0\]
in the regime for $\varepsilon, \delta, \mu$ given in Item (2), which is nowadays called the Kawahara equation.

\medskip

Our main focus is to investigate the higher-order extension of KdV and Kawahara equations. Consider the Cauchy problem for the following higher-order KdV-type equation posed on the unit circle $\mathbb{T}$:
\begin{equation}\label{eq:kdv_int}
\begin{cases}
\pt u + (-1)^{j+1} \px^{2j+1}u + \frac12 \px(u^2)=0,\\
u(0,x) = u_0(x) \in H^s(\T),
\end{cases}
\qquad (t,x) \in \R \times \T, 
\end{equation}
for $j \in \MB{N}$ and $u$ is a real-valued function.  Especially, \eqref{eq:kdv_int} is called KdV and Kawahara equation when $j=1$ and $j=2$, respectively. These types of equations have conservation laws such as  
\begin{equation}\label{laws}
\begin{split}
M[u]&= \int_{\MB{T}} u \; dx, \quad  \text{(Mass)} \\
E[u] &= \int_{\MB{T}} u^2 \; dx, \\
H[u] &= \int_{\MB{T}} \frac{1}{2} \left(\partial_x^j u\right) ^2 - \frac{1}{6} u^3 \; dx,\quad\text{(Hamiltonian)}.
\end{split}
\end{equation}
Furthermore, \eqref{eq:kdv_int} is the Hamiltonian equation with respect to $H[u]$ defined in \eqref{laws}. In other words, we can rewrite \eqref{eq:kdv_int} as follows:
\begin{equation*}
u_t = \partial_x \nabla_u H\left(u\left(t\right)\right) = \nabla_{\omega} H\left(u\left(t\right)\right)
\end{equation*}
where $\nabla_u$ is the $L^2$ gradient and $\nabla_{\omega} = \nabla_{\omega_{-\frac{1}{2}}}$ is the symplectic gradient 
\begin{equation*}
\omega_{-\frac{1}{2}}\left(u,v\right) := \int_{\MB{T}} u \partial_x^{-1} v dx.
\end{equation*}
These three conservation laws play various roles (in particular, in particular, to determine the global behavior of solutions and the global control properties of equation \eqref{eq:kdv_int}) in the study of the partial differential equations. 

\subsection{Problems under consideration} In this paper, we prove that the higher-order KdV-type equation\footnote{One may generalize the equation \eqref{eq:kdv} as
\[\pt u + \sum_{m=0}^j \alpha_m \px^{2m+1}u + \frac12 \px(u^2)=f,\]
where $\alpha_m \in \R$. However, main analyses in the paper are almost analogous without additional difficulties, thus the equation \eqref{eq:kdv} does not lose the generality in a sense of the aim in this paper. See Remark \ref{rem:restriction}.}
\begin{equation}\label{eq:kdv}
\begin{cases}
\pt u + (-1)^{j+1} \px^{2j+1}u + \frac12 \px(u^2)=f(t,x),\\
u(0,x) = u_0(x) \in H^s(\T),
\end{cases}
\qquad (t,x) \in \R \times \T, 
\end{equation}
posed on periodic domain $\mathbb{T}$ is globally controllable in $H^s$, for $s\geq0$, when we introduce a forcing term $f=f(t,x)$ added to the equation as a control input. Here, $f$ is assumed to be supported in a given open set $\omega
\subset\mathbb{T}$. The following control problems are considered:

\medskip

\noindent\textbf{Exact control problem}: \textit{Given an initial state
}$u_{0}$\textit{ and a terminal state }$u_{1}$\textit{ in a certain space, can
one find an appropriate control input }$f$\textit{ so that the equation \eqref{eq:kdv} admits a solution }$u$\textit{ which satisfies} $\left.u\right\vert _{t=T}=u_{1}$?

\medskip

\noindent\textbf{Stabilization problem}: \textit{Can one find a feedback
control law} $f=Ku$ \textit{so that the resulting closed-loop system}
\[\pt u + (-1)^{j+1} \px^{2j+1}u + \frac12 \px(u^2)=Ku,\quad (t,x)\in\mathbb{R}\times\mathbb{T}, \]
\textit{is asymptotically stable at an equilibrium point as $t\rightarrow+\infty$?}

\medskip

The higher-order KdV-type equations keep its mass conserved, see for instance \eqref{laws}, thus
\[
\frac{d}{dt}\int_{\mathbb{T}}u (t,x) \; dx = 0,
\]
for any $t\in\mathbb{R}$ when no control is in action ($f\equiv0$). In
applications, one would also  like to keep the mass conserved while conducting
control. For that purpose, a natural constraint on
our control input $f$ is as follows:%
\[
\int_{\mathbb{T}}f(t,x) \; dx=0\text{, }\forall t\in\mathbb{R}
\text{.}%
\]
Thus, as in \cite{Russel96}, the natural control input $f(t,x)$ is chosen to be of the
form%
\begin{equation}
f(t,x)  = [Gh] (t,x)  :=g(x)  \left(  h(t,x)  -\int_{\mathbb{T}}g(y) h(t,y) \; dy\right),  \label{k5a}
\end{equation}
where $h$ is considered as a new control input, and $g(x)$ is a given nonnegative smooth function such that
\[
2\pi\left[  g\right]  =\int_{\mathbb{T}}g\left(  x\right)  dx=1\text{.}%
\]
Here, we denote $\omega$ by the set $\omega :=\left\{  g>0\right\}$, where the control function is effectively acting.

\subsection{Review of the results in the literature} 
The local and global well-posedness of \eqref{eq:kdv_int} were widely studied. For the local well-posedness result, Gorsky and Himonas \cite{Gorsky:2009eg} firstly proved this problem for $s\ge-\frac12$ and Hirayama \cite{Hirayama} improved for $s \ge -\frac{j}{2}$. Both works are based on the standard Fourier restriction norm method. Hirayama improved the bilinear estimate by using the factorization of the resonant function.

The results of the global well-posedness for \eqref{eq:kdv_int}, when $j=1,2$, were proved by Colliander et al. \cite{CKSTT1} and Kato \cite{Kato}, respectively, via "I-method". In \cite{HongKwak} the authors extend the results of \cite{CKSTT1} and \cite{Kato} for $j\ge3$. The method basically follows the argument in \cite{CKSTT1} for periodic KdV equation, while some estimates are slightly different. More precisely, they showed that for $j \ge 3$ and $s \ge -\frac{j}{2},$ the IVP \eqref{eq:kdv_int} is globally well-posed in $H^s(\T)$.

Regarding the control theory, when $j=1$, the system \eqref{eq:kdv} has good control properties. The study of the controllability and stabilization to the KdV equation started with the work of Russell and Zhang  \cite{RusselZhang1993}  for the linear system 
\begin{equation}
u_{t}+u_{xxx}=f\text{, }
\label{I5}%
\end{equation}
with periodic boundary conditions and an internal control $f$. Since then, both controllability and stabilization problems have been intensively studied. 

It is well-known that \eqref{I5} with $f = -uu_x$ allows an infinite set of conserved integral
quantities, for instance, $M[u]$ and $E[u]$, defined in \eqref{laws}.  From the historical origins of the KdV equation involving the behavior of water waves in a shallow channel \cite{Boussinesq,Korteweg,Miura}, it is natural to think of $M[u]$ and $E[u]$ as expressing
conservation of \textit{volume (or mass)} and \textit{energy}, respectively.

The Russell and Zhang's work  \cite{RusselZhang1993} is purely linear.  In fact, until Bourgain \cite{Bourgain1993-1} discovered a subtle smoothing property of solutions of the KdV equation posed on a periodic domain, no results of the nonlinear problems were solved. This novelty, discovered by Bourgain, has played a crucial role in the proof of the results in \cite{Russel96}.

Specifically, in \cite{Russel96} the authors studied the nonlinear equation associated to \eqref{I5} from a control point of
view with a forcing term $f=f(t,x)$ added to the equation as a control input:
\begin{equation}
u_t + u_{xxx} + uu_x = f, \quad (t,x) \in \R \times \T.
\label{I6}%
\end{equation}
With this in hand, Russell and Zhang were able to show the local exact controllability and local exponential stabilizability for the system \eqref{I6}. Indeed, the results presented in  \cite{Russel96} are essentially linear; they are more or less small perturbations of the linear results. After these works, Laurent \textit{et al.} \cite{Laurent} show that still it is possible to guide the system \eqref{I6} from a given initial state $u_{0}$ to a given terminal state $u_{1}$ when $u_{0}$ and $u_{1}$ have large amplitude by choosing an appropriate control input. Furthermore, they showed that the large amplitude solutions of the closed-loop system \eqref{I6} decay exponentially as $t \to \infty$. Hence, the authors in \cite{Laurent} proved global exact controllability and global exponential stabilizability extending the results obtained by Russell and Zhang in \cite{Russel96}. These global results are established with the aid of certain of propagation of compactness and regularity in Bourgain spaces for the solutions of the associated linear system of \eqref{I6}.

Considering $j=2$ the system \eqref{eq:kdv} is the so-called Kawahara equation
\begin{equation}\label{kawahara}
\pt u -\px^5u + \frac12 \px(u^2)=f,  \qquad (t,x) \in \R \times \T.
\end{equation}
Recently, the first author, in \cite{CaKawahara}, studied the stabilization problem and conjectured a critical set phenomenon for Kawahara equations as occurs with the KdV equation \cite{CaZh,Rosier} and Boussinesq KdV-KdV system \cite{CaPaRo1}, for example. Moreover, as far as we know, the control problem was, first, studied in \cite{zhang,zhang1} when the authors considered the Kawahara equation on a periodic domain $\mathbb{T}$ with a distributed control of the form \eqref{k5a}. First, the authors were able to prove the local controllability results for this equation in \cite{zhang}. Aided by smoothing properties of the system in  Bourgain spaces, they were able to show that the Kawahara equation  is globally exactly controllable and globally exponentially stabilizable (see \cite{zhang1}).

We caution that this is only a small sample of the extant equations with the similar structure to the system \eqref{eq:kdv}, \eqref{I5} and \eqref{kawahara}. For an extensive review of the physical meanings of these equations, as well as well-posedness and controllability results the authors suggest the following nice references \cite{CaGomes,cerpa,Hirayama,Kwak2019} and the references therein.

\subsection{Notation and main results} Let us introduce some notation and present the main results of the manuscript.  Let us consider the Fourier and inverse Fourier transforms with respect to the spatial variable $x \in \T$,
\[\mathcal F_x(f)(k) = \widehat{f}(k) := \frac{1}{\sqrt{2\pi}}\int_{\T} e^{-ikx}f(x) \; dx \quad \mbox{and} \quad \mathcal F_x^{-1}(f)(x) = \frac{1}{\sqrt{2\pi}}\sum_{k \in \Z} e^{ikx}f(k),\]
respectively. Additionally, the space-time Fourier and inverse Fourier transforms are
\[\mathcal F(f)(\tau, k) = \widetilde{f}(\tau, k) := \frac{1}{2\pi}\int_{\R \times \T} e^{-it\tau}e^{-ikx}f(x) \; dx\;dt\]
and
\[\mathcal F^{-1}(f)(t,x) = \frac{1}{2\pi}\int_\R\sum_{k \in \Z} e^{it\tau}e^{ikx} f(\tau,k) \; d\tau,\]
respectively.

Consider now the $H^s(\mathbb{T}):=H^s$ space with the inner product as
\[(f,g)_{H^s} = (f,g)_s := \sum_{k\in \Z} \langle k \rangle^{2s} \widehat{f}(k)\overline{\widehat{g}(k)},\]
where $\langle \cdot \rangle = (1+|\cdot|^2)^{\frac12}$. We simply denote the $H^0:=L^2$ inner product by $(\cdot, \cdot)$. It naturally defines $H^s$ norm as $\norm{f}_{H^s} = \sqrt{(f,f)_{H^s}}$. We will use $H^s_0(\T)$ as the subspace of $H^s(\T)$ whose elements obey the mean zero condition, i.e.,
\[H_0^s(\T) = \left\{f \in H^s(\T) : \int_{\T} f = 0\right\}.\]

The aim of this manuscript is to address the control and stabilization (particularly global) issues. However, before presenting the global results, let us present a theorem that shows the exact control result.

\begin{theorem}[\cite{ZB2018}]
\label{control_zhang}Let $T>0$ and $s\geq0$ be given. There exists a
$\delta>0$ such that for any $u_{0}$, $u_{1}\in H^{s}\left(  \mathbb{T}%
\right)  $ with
\[
\left\Vert u_{0}\right\Vert _{H^{s}\left(  \mathbb{T}\right)  }\leq
\delta\quad\text{and}\quad\left\Vert u_{1}\right\Vert _{H^{s}\left(  \mathbb{T}\right)
}\leq\delta\text{,}%
\]
one can find a control function $h\in L^{2}\left(  \left[  0,T\right]
;H^{s}\left(  \mathbb{T}\right)  \right)  $ such that the system%
\begin{equation}
\pt u + (-1)^{j+1} \px^{2j+1}u + \frac12 \px(u^2)=Gh\text{, } \quad (t,x) \in \R \times \T,
\label{k5}
\end{equation}
where $G$ is defined by \eqref{k5a}, admits a solution $u\in C\left(  \left[  0,T\right]  ;H^{s}\left(
\mathbb{T}\right)  \right)  $ satisfying%
\[
\left.  u\right\vert _{t=0}=u_{0}\text{, \ }\left.  u\right\vert _{t=T}%
=u_{1}\text{.}%
\]
\end{theorem}

Now, thanks to the advantage of the results proved in \cite{Liu,Slemrod}, the local exponential result in $H^s(\mathbb{T})$, for any $s\geq0$, can be established.

\begin{theorem}[\cite{ZB2018}]
\label{stability_zhang}Let $s\geq0$ and $\lambda>0$ be given. There exists a
bounded linear operator $$K_{\lambda}:H^{s}\left(  \mathbb{T}\right)
\rightarrow H^{s}\left(  \mathbb{T}\right) $$ such that if one chooses the
feedback control $h=K_{\lambda}u$ in \eqref{k5}, then the resulting closed-loop system%
\begin{equation}
\begin{cases}
\pt u + (-1)^{j+1} \px^{2j+1}u + \frac12 \px(u^2)=GK_{\lambda}u, \\
u(0,x)  =u_0(x),
\end{cases}
\qquad (t,x) \in \R \times \T,  \label{k6}%
\end{equation}
is locally exponentially stable in the space $H^{s}\left(  \mathbb{T}\right)
$, for $s\geq0$, that is, there exists a $\delta>0$ such that for any $u_{0}\in H^{s}\left(
\mathbb{T}\right)  $ with $
\left\Vert u_{0}\right\Vert _{H^{s}\left(  \mathbb{T}\right)  }\leq
\delta\text{,}$
the corresponding solution $u$ of \eqref{k6} satisfies%
\[
\left\Vert u\left(  \cdot,t\right)  -\left[  u_{0}\right]  \right\Vert
_{H^{s}\left(  \mathbb{T}\right)  }<Ce^{-\lambda t}\left\Vert u_{0}-\left[
u_{0}\right]  \right\Vert _{H^{s}\left(  \mathbb{T}\right)  }\text{,}\quad \forall t>0.
\]
\end{theorem}

\begin{remark}
We point out that Theorems \ref{control_zhang} and \ref{stability_zhang} have already been proved by Zhao and Bai \cite{ZB2018}. For self-containedness, we will also give rigorous proofs of them in Appendices \ref{sec3}  and \ref{sec5}.
\end{remark}

These results shown that one can always find an appropriate control input $h$ to guide the system \eqref{k5} from a given initial state $u_{0}$ to a given terminal state
$u_{1}$ as long as \textit{their amplitudes are small} and $\left[
u_{0}\right]  =\left[  u_{1}\right]  $. However, some natural questions arise.

\vspace{0.2cm}

\noindent\textbf{Question $\mathcal{A}$}. Can one still guide the system   \eqref{k5} by choosing an appropriate control input $h$ (defined on a sufficiently long time interval)
from a given initial state $u_{0}$ to a given terminal state $u_{1}$ when
$u_{0}$ or $u_{1}$ have large amplitude?

\vspace{0.2cm}

According to Theorem \ref{stability_zhang}, solutions of system (\ref{k5})
issued from initial data close to their mean values converge at a uniform
exponential rate to their mean values in the space $H^{s}\left(
\mathbb{T}\right)  $ as $t\rightarrow+\infty$. One may ask the following issue:

\vspace{0.2cm}

\noindent\textbf{Question $\mathcal{B}$.} Does any solution of the closed-loop system
\eqref{k6} converge exponentially to its mean value as $t\rightarrow+\infty$?

\vspace{0.2cm}

Thus, additionally to the local results, presented in Theorems \ref{control_zhang}  and \ref{stability_zhang}, our work gives a positive answer to these questions that have a global character. This is possible thanks to the celebrated results obtained by Bourgain \cite{Bourgain1993-1}.  One of the main results in this work gives an answer to the Question $\mathcal{A}$, the result ensures that the system \eqref{k5} is \textit{globally} exactly controllable.

\begin{theorem}\label{control_global}
Let $s\geq0$, $R>0$ and $\mu\in\mathbb{R}$ be given.
There exists a time $T>0$ such that if $u_{0}$, $u_{1}\in H^{s}\left(
\mathbb{T}\right)  $, with $\left[  u_{0}\right]  =\left[  u_{1}\right]  =\mu$,
satisfies
\[
\left\Vert u_{0}\right\Vert _{H^{s}\left(  \mathbb{T}\right)  }\leq R\text{,
\ \ }\left\Vert u_{1}\right\Vert _{H^{s}\left(  \mathbb{T}\right)  }\leq
R\text{,}%
\]
then one can find a control input $h\in L^{2}\left(  0,T;H^{s}\left(
\mathbb{T}\right)  \right)  $ such that the system (\ref{k5}) admits a
solution $u\in C\left(  \left[  0,T\right]  ,H^{s}\left(  \mathbb{T}\right)
\right)  $ satisfying%
\[
\left.  u\right\vert _{t=0}=u_{0}\text{, \ }\left.  u\right\vert _{t=T}%
=u_{1}\text{.}%
\]
\end{theorem}

As for Question $\mathcal{B}$, we have the following affirmative answer.

\begin{theorem}\label{exponential_global}
Let $s\geq0$ and $\mu\in\mathbb{R}$. There exists a
constant $\gamma>0$ such that for any $u_{0}\in H^{s}\left(  \mathbb{T}%
\right)  $ with $\left[  u_{0}\right]  =\mu$, the corresponding solution $u$
of the system \eqref{k5}, with $h\left(  x,t\right)  =-G^{\ast}u\left(
x,t\right)  $, satisfies%
\[
\left\Vert u\left(  \cdot,t\right)  -\left[  u_{0}\right]  \right\Vert
_{H^{s}\left(  \mathbb{T}\right)  }<\alpha_{s,\mu}\left(  \left\Vert
u_{0}-\left[  u_{0}\right]  \right\Vert _{L^{2}\left(  \mathbb{T}\right)
}\right)  Ce^{-\gamma t}\left\Vert u_{0}-\left[  u_{0}\right]  \right\Vert
_{H^{s}\left(  \mathbb{T}\right)  }\text{, }\forall t\geq0\text{,}%
\]
where $\alpha_{s,\mu}:\mathbb{R}^{+}\rightarrow\mathbb{R}^{+}$ is a
nondecreasing continuous function depending on $s$ and $\mu$.
\end{theorem}

\subsection{Heuristic of the article} In this manuscript our goal is to give answers for two global control problems mentioned in the previous section. Observe that the results obtained so far are concentrated in a single KdV equation \eqref{I6}, see e.g. \cite{RusselZhang1993,Russel96}, and Kawahara equation \eqref{kawahara}, see for instance \cite{zhang,zhang1}. Moreover, even higher-order KdV type \eqref{eq:kdv} has been studied in the sense of local controllability and stabilization in \cite{ZB2018}, while global control results for the generalized higher-order KdV type equation \eqref{eq:kdv} are still open, so, under this direction, our work is a generalization of the previous result for KdV and Kawahara equations. Let us describe briefly the main arguments of the proof of our theorem and give consideration of the importance of the work in the study of the control theory for general dispersive operators.

The first two results are local, that is, Theorems \ref{control_zhang}  and \ref{stability_zhang}.  In fact, first, thanks to the properties of the operator 
\begin{equation}\label{eq:A}
Aw=-(-1)^{j+1}\partial^{2j+1}_x w,
\end{equation}
we can use classical theorems of Ingham and Beurling \cite{Ing1936,Beu} to ensures that the linear system associated to \eqref{k5} have control and stabilization properties. To extend these results for the nonlinear case, one has to control one derivative in the nonlinear term. However, it is well-known that linear solutions have no dispersive and no smoothing effect under periodic boundary condition. Thanks to Bourgain \cite{Bourgain1993-1, Bourgain1993-2}, by regarding the linear estimate as multilinear interactions in $L^2$, now, one can recover derivative loss occurring in Sobolev inequality, thus the nonlinearity can be controlled. Note that this is not the only way to handle the nonlinearity, compare \cite{Hirayama} with \cite{HongKwak}. Here, the main point is to prove the following \textit{Strichartz estimates} 
$$ \norm{f}_{L^4(\R \times \T)} \lesssim \norm{f}_{X^{0,b}}, \quad b > \frac{j+1}{2(2j+1)},\quad \text{for all} \;j\in \mathbb{N}.  $$ 
After that we are able to extend the local solutions for the global one and prove the nonlinear (local) control results as a perturbation of the linear one. Note that the arguments here are purely linear. In addition, we emphasize that the exact controllability and stabilizability results of the linear system associated to \eqref{k5} are valid in $H^{s}(\mathbb{T})$ for any $s \in \mathbb{R}$.

It is important to point out that Theorems \ref{control_global} and \ref{exponential_global} have global character. Precisely, the control result for large data (Theorem \ref{control_global}) will be a combination of a global stabilization result (Theorem \ref{exponential_global}) and the local control result (Theorem \ref{control_zhang}). Indeed, given the initial state $u_{0} \in H^{s}(\mathbb{T})$ to be controlled, by means of the damping term $Ku =-GG^{\ast}u$ supported in $\omega\subset\mathbb{T}$, i.e., solving the IVP \eqref{k5} with $h=-G^{\ast}u$, we drive $u_{0}$ to a state $\tilde{u}_{0}$ close enough to the mean value $\mu$ in a sufficiently large time $T_{1}$, by Theorem \ref{exponential_global}. Again, using this result, we do the same with the final state $u_{1} \in H^{s}(\mathbb{T})$  by solving the system backwards in time, thanks to the time reversibility of the higher-order KdV-type equation. This process produces two states $\tilde{u}_{0}$ and $\tilde{u}_{1}$ which are close enough to $\mu$ so that the local controllability result (Theorem \ref{control_zhang}) can be applied around the state $u(x)=\mu$. We can see this mechanism illustrated in the Figure \ref{Globalcontrol}.

\begin{figure}[h!]
		\centering
		\includegraphics[width=0.74\linewidth]{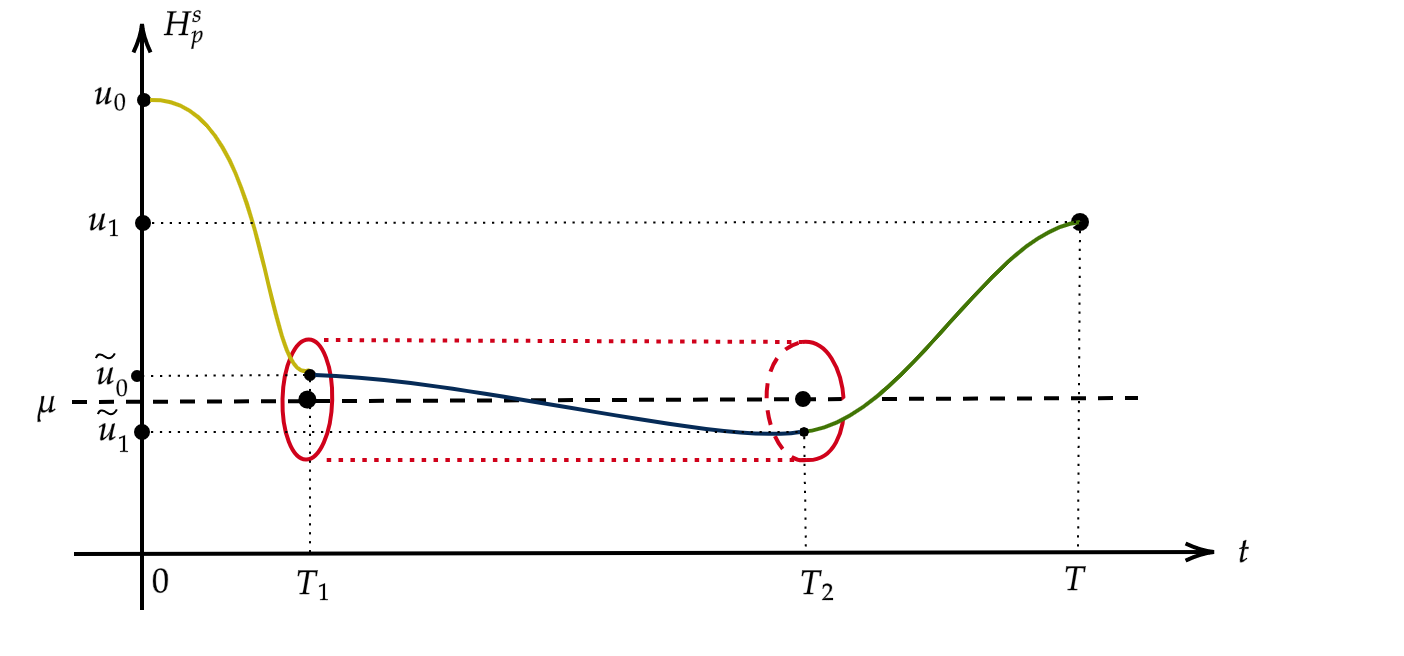}
		\caption{The constructive approach of the proof of the Theorem \ref{control_global}.}
		\label{Globalcontrol}
	\end{figure}

Lastly, the proof of the Theorem \ref{exponential_global} is equivalent to prove an o\textit{observability inequality}, which one, by using contradiction arguments, relies on to prove a unique continuation property for the system \eqref{k5}. This property is achieved thanks to the propagation results using again smooth properties of solution in Bourgain spaces. The main difficulties to prove the propagation results arises from the fact that the system \eqref{k5} has a general structure. To overcome this difficulty the Strichartz estimates, for the solution of our problem in Bourgain Spaces, are essential.

We finish this introduction by mentioning that the global results presented in this article, even the local results, are not a consequence of the previous results for the KdV and Kawahara equations. Indeed, taking $j=1$ and $2$ in our operator $A$, defined in \eqref{eq:A}, we can recover the previous results in the literature for these equations,  nevertheless, the necessary estimates to treat the operator $ A $ as well as the propagation of singularities are the main novelties of this work. In summary, the key ingredients of this work are:
\begin{itemize}
\item[$1.$] \textit{Strichartz estimates} associated to the solution of the problem under consideration;
\item[$2.$] \textit{Microlocal analysis} to prove propagation of the regularity and compactness;
\item[$3.$] \textit{Unique continuation property} for the operator $A$.
\end{itemize} 

\subsection{Structure of the paper} Some preliminaries are given in Section \ref{sec:pre}, particularly, spectral property of the operator $A$ in \eqref{eq:A} is studied and Bourgain spaces are introduced. In Section \ref{sec:L^4}, we give a rigorous proof of Strichartz estimate. In Section \ref{Apendice}, we investigate propagation of singularities and unique continuation property. The global stabilization result (Theorem \ref{exponential_global}) is proved in Section \ref{sec6}, and in Section \ref{sec7}, some comments and open questions are presented. In Appendices, as mentioned, some analyses for local results are given. The linear system is studied in Appendix \ref{sec3}, particularly, we present the linear control problems, which are a consequence of the spectral analysis. A brief proof of the global well-posedness of the closed loop system is presented in Appendix \ref{sec4}. Finally the proofs of Theorems \ref{control_zhang} and \ref{stability_zhang} are presented in Appendix \ref{sec5}.

\section{Preliminaries}\label{sec:pre}
\subsection{Spectral property}

Consider the operator $A$ denoted by
\begin{equation}\label{eq:operator}
Aw=-(-1)^{j+1}\partial^{2j+1}_x w
\end{equation}
with domain $\mathcal{D}\left(  A\right)  =H^{2j+1}\left(  \mathbb{T}\right)
$. The operator  $A$ generates a continuous unitary operator group $W(t)$ on the space $L^{2}%
(\mathbb{T})$, precisely,
\[W(t)f(x) = \frac{1}{2\pi}\sum_{n \in \Z} e^{ikx} e^{itk^{2j+1}} \widehat{f}(k).\]
Remark that
\begin{equation}\label{eq:AW}
A^* = -A \quad \mbox{and} \quad W^*(-t) = W(t).
\end{equation}
One immediately knows that eigenfunctions of the operator $A$ are the orthonormal Fourier bases of $L^{2}(\mathbb{T})$,%
\[
\phi_{k}\left(  x\right)  =\frac{1}{\sqrt{2\pi}}e^{ikx}, \quad k \in \Z,
\]
and its corresponding eigenvalues are
\begin{equation}\label{eq:EV}
\lambda_{k}=i k^{2j+1}\text{, } \quad \quad k \in \Z.
\end{equation}

We now prove a \textit{gap condition} which will be used to prove a local controllability result in Section \ref{sec3}. The result can be read as follows.
\begin{lemma}[Gap condition]\label{lem:gap}
Let $j \ge 1$ and $k \in \Z$. For $\lambda_k$ defined as in \eqref{eq:EV}, if $|k| \ge j +1$, we have
\begin{equation}\label{eq:gap2-1}
|\lambda_{k+1} - \lambda_k| \ge k^2,
\end{equation}
which implies
\begin{equation}\label{eq:gap2-2}
\lim_{|k| \rightarrow +\infty} |\lambda_{k+1} - \lambda_k| =\infty.
\end{equation}
\end{lemma}
\begin{proof}
When $j=1$, it is easy to see that
\[\begin{aligned}
-i(\lambda_{k+1} - \lambda_k) =&~{} 3k^2 + 3k + 1 = \frac32k^2 + \frac32(k^2+2k+1) - \frac12 \\
=&~{} k^2 + \frac32(k+1)^2 +\frac12(k^2-1)\ge ~{}k^2,
\end{aligned}\]
for  $|k| \ge 1$,which satisfies \eqref{eq:gap2-1}, and thus \eqref{eq:gap2-2} follows.

Now, fix $j \ge 2$. A straightforward computation yields
\[2m k^2 + (2j+1-2m+1) k > 0, \quad m=1,2,\cdots,j,\]
whenever $k > 0$ or $k < - \frac{j-m+1}{m}$, which implies
\[\begin{aligned}
\binom{2j+1}{2m-1}k^2 &+ \binom{2j+1}{2m}k\\
=&~{} \left(2m k^2 + (2j+1-2m+1) k\right) \frac{(2j+1)(2j)\cdots(2j+1-(2m-1)+1)}{(2m)!}\\
>&~{}\frac{(j+1)}{m}\binom{2j+1}{2m-1}.
\end{aligned}\]
Thus, we conclude for $|k| \ge j+1$ that
\[\begin{aligned}
| \lambda_{k+1} - \lambda_k| =&~{} \left|\sum_{\ell=1}^{2j+1} \binom{2j+1}{\ell} k^{2j+1-\ell} \right|\\
=&~{} \left|1 + \sum_{m=1}^{j} \left(\binom{2j+1}{2m-1} k^2+\binom{2j+1}{2m} k\right)k^{2(j-m)} \right|\\
>&~{} (2j+1)(j+1)k^{2(j-1)}  \\
>&~{} k^2,
\end{aligned}\]
which completes the proof.
\end{proof}
\begin{remark}\label{rem:restriction}
Lemma \ref{lem:gap} ensures that the gap condition is still valid for the (generalized) linear operator
\[\tilde A = \sum_{m=0}^{j} (-1)^m\alpha_m\partial_x^{2m-1},\]
where $\alpha_m \ge 0$ with $\alpha_j \neq 0$. Indeed, we have eigenvalues associated to $\tilde A$ as
\[\tilde{\lambda}_k = i\sum_{m=0}^j \alpha_m k^{2m-1},\]
and the gap condition
\[|\tilde \lambda_{k+1} - \tilde \lambda_k| \ge k^2\max_{m=0,\cdots, j}\{\alpha_m\},\]
for $|k| > j$, which gives  $|\tilde \lambda_{k+1} - \tilde \lambda_k| \rightarrow \infty$ when $|k| \to \infty$. Additionally, it is not necessary to restrict $\alpha_m \ge 0$, however we do not further discuss about it here.
\end{remark}

\subsection{Fourier restriction spaces}

The function space equipped with the Fourier restriction norm, which is the so-called $X^{s,b}$ spaces, has been proposed by Bourgain \cite{Bourgain1993-1,Bourgain1993-2} to solve the periodic NLS and generalized KdV. Since then, it has played a crucial role in the theory of dispersive equations, and has been further developed by many researchers, in particular,  Kenig, Ponce and Vega \cite{KPV1996} and Tao \cite{Tao2001}. In the following, to ensure global control results, the space $X^{s,b}$ will be of paramount importance.

Let $f$ be a Schwartz function, i.e., $f \in \mathcal{S}_{t,x}(\R \times \T)$. $\wt{f}$ or $\ft (f)$ denotes the space-time Fourier transform of $f$ defined by
\[\wt{f}(\tau,n)=\frac{1}{2\pi}\int_{\R}\int_{0}^{2\pi} e^{-ixn}e^{-it\tau}f(t,x)\; dxdt .\]
Then, it is known that the (space-time) inverse Fourier transform is naturally defined as
\[f(t,x)=\frac{1}{2\pi}\sum_{n\in\Z}\int_{\R} e^{ixn}e^{it\tau}\wt{f}(\tau,n)\;dt .\]
Moreover, we use $\ft_x$ (or $\wh{\;}$ ) and $\ft_t$ to denote the spatial and temporal Fourier transform, respectively.

For given $s,b\in\mathbb{R}$, we define the space $X^{s,b}$ associated to \eqref{eq:kdv} as the closure of $\mathcal{S}_{t,x}(\R \times \T)$ under the norm
\[\norm{f}_{X^{s,b}}^2 = \frac{1}{2\pi}\sum_{k \in \Z} \int_{\R} \bra{k}^{2s} \bra{\tau-k^{2j+1}}^{2b}|\widetilde{f}(\tau,k)|^2 \; d\tau,\]
which is equivalent to the expression $\norm{W(-t)f(t,x)}_{H_t^bH_x^s}$. Note that the definition of $X^{s,b}$ ensures the trivial nesting property
\begin{equation}\label{eq:nesting}
X^{s,b} \subset X^{s',b'}, \quad \mbox{whenever} \quad s' \le s,~ b' \le b,
\end{equation}
and this immersion is continuous.
Moreover, it is known (see, for instance, \cite[Lemma 2.11]{Tao}) that $X^{s,b}$ space is stable with respect to time localization, that is,
\begin{equation}\label{eq:Xsb stable}
\norm{\eta(t) u}_{X^{s,b}} \lesssim_{b, \psi} \norm{u}_{X^{s,b}},
\end{equation}
for any time cutoff function $\eta \in \mathcal S_t(\R)$.

According to \cite[Theorem 1.2]{KPV1996}, it is necessary to fix the exponent $b =\frac12$ in $X^{s,b}$ for the study of the periodic KdV equation. Otherwise, one cannot, indeed, obtain one-derivative gain in the \emph{high-low} non-resonant interactions to kill the derivative in the nonlinearity. Thereafter, it becomes natural to fix $b=\frac12$ even for the other periodic problems, but it is not necessary. For instance, in the higher-order KdV-type case, in particular, $j \ge 2$, one can obtain $\min\{2bj, (1-b)j\}$-derivative gains from the \emph{high-low} non-resonant interactions, which is enough to remove the one derivative in the nonlinearity, whenever $\frac{1}{2j} \le b \le 1- \frac{1}{2j}$. However, the present paper covers the KdV and Kawahara cases as well, we, thus, fix $b = \frac12$, throughout the paper.

On the other hand, $X^{s,\frac12}$ space does not be embedded in the classical solution space $C_tH^s\equiv C(\mathbb{R};H^s(\mathbb{T}))$, it, thus, itself is not enough for the well-posedness theory. To complement the lack of  the embedding property, we define the space $Y^s$ for solutions with the following norm
\[\norm{f}_{Y^s} := \norm{f}_{X^{s,\frac12}} + \big\|\bra{k}^{s}\widetilde f\big\|_{\ell_k^2L_{\tau}^1}.\]

For a given time interval $I$, let $X_I^{s,b}  $ (resp. $Y^s_I$) denote the time localization of $X^{s,b}$ (resp. $Y^s$) on the interval $I$ with the norm
\[
\norm{f}_{X_I^{s,b}}=\inf \left\{\norm{g}_{X^{s,b}} : g=f \; \mbox{ on } \; I \times \T\right\}
\]%
\[
\left(  \text{ resp. }\norm{f}_{Y_I^s}=\inf \left\{\norm{g}_{Y^s} : g=f \; \mbox{ on } \; I \times \T\right\}  \right)  .
\]
For simplicity, we denote $X_I^{s,b}  $ (resp. $Y^s_I$) by $X_T^{s,b}$ (resp. $Y_T^s$), if $I= (0,T)$.

\subsection{Estimates for higher-order KdV-type equation} We summarize well-known estimates, already established in the literature, which will play important roles in establishing the exact controllability and stabilizability of the system \eqref{eq:kdv}. For this, we introduce a cut-off function $\eta \in C_{c}^{\infty}(\mathbb{R})$ such that $\eta=1,$ if $t\in[-1,1]$ and $\eta=0,$ if $t\notin (-2,2).$

\begin{lemma}[$X^{s,b}$ estimates, \cite{GTV1997, Hirayama}]\label{lem:Xsb}
Let  $0 < T < \infty$ be given. Then,
\begin{enumerate}
\item For all $s \in \R$, we have for $u \in Y_T^s$
\[\norm{u}_{C_TH^s} \lesssim \norm{u}_{Y_T^s}.\]
\item For all $s \in \R$, we have for $f \in H^s$
\[\norm{W(t)f}_{Y_T^s} \lesssim_{T,\eta} \norm{f}_{H^s}.\]
If $T\leq 1,$ then the constant in the right-hand side does not depend on $T.$

\item For all $s \in \R$, we have for $F \in Y_T^s$
\begin{equation}\label{eq:inhomogeneous}
\left\|\int_0^t W(t-\tau)F(\tau)\; d\tau \right\|_{Y_T^s} \lesssim_{T,\eta} \norm{\mathcal F^{-1}\left(\bra{\tau - k^{2j+1}}^{-1}\widetilde F\right)}_{Y_T^s}.
\end{equation}
If $T\leq 1,$ then the positive constant involved in \eqref{eq:inhomogeneous} does not depend on $T.$

\item For $s\in \mathbb{R}$ with $s \ge -\frac{j}{2}$, we have for $u,v \in Y_T^s$
\[\norm{\mathcal F^{-1}\left(\bra{\tau - k^{2j+1}}^{-1}\widetilde{\partial_x(uv)}\right)}_{Y_T^s} \lesssim \norm{u}_{Y_T^s}\norm{v}_{Y_T^s}\]

\item For all $s \in \R$, $-\frac12 < b' \le b < \frac12$ and $0 < T < 1$, we have for $u \in X_T^{s,b}$
\[\norm{u}_{X_T^{s,b'}} \lesssim T^{b-b'} \norm{u}_{X_T^{s,b}}.\]
\end{enumerate}
\end{lemma}

\begin{remark}\label{rem:nonlinear embedding}
The right-hand side of \eqref{eq:inhomogeneous} is simply dominated by $\norm{F}_{L^2(0,T;H^s)}$ due to the definition of $Y_T^s$ norm, the nesting property \eqref{eq:nesting} and the weight $\bra{\tau - k^{2j+1}}^{-1}$.
\end{remark}

\section{$L^4$-Strichartz estimate}\label{sec:L^4}
In this section, we provide a rigorous proof of $L^4$-Strichartz estimate for higher-order KdV equation. 

\begin{lemma}[Strichartz estimates]\label{strichartz_estimates}
The following estimates
hold:%
\begin{equation}
\norm{f}_{L^4(\R \times \T)} \lesssim \norm{f}_{X^{0,b}}, \quad b > \frac{j+1}{2(2j+1)},\quad \forall j\in \mathbb{N}, \label{k11}
\end{equation}
where the implicit constant depends on $b$ and $j$.
\end{lemma}

\begin{remark}
As well-known, the intuition of Lemma \ref{strichartz_estimates} is as follows: By Sobolev embedding (in both time and spatial variables), one has
\[\|f\|_{L_{t,x}^4} \lesssim \|S(-t)f\|_{H_t^{\frac14}H_x^{\frac14}}.\]
On the other hand, from \eqref{eq:kdv_int}, one roughly guesses that $\partial_t \sim \partial_x^{2j+1}$, which transfers spatial derivatives to temporal derivatives ($\partial_x^{\frac14} \mapsto \partial_t^{\frac{1}{4(2j+1)}}$). Hence, one can guess  
\[\|f\|_{L_{t,x}^4} \lesssim \|S(-t)f\|_{H_t^{\frac14}H_x^{\frac14}} \lesssim \|S(-t)f\|_{H_t^{\frac14 + \frac{1}{4(2j+1)}}L_x^2} \lesssim \|f\|_{X^{0,b}}, \quad b > \frac{j+1}{2(2j+1)}.\]
The equality $b =  \frac{j+1}{2(2j+1)}$ can also be obtained, but we do not attempt to give it here, in order to avoid complicated computations. 
\end{remark}

The following lemma plays an essential role to prove Lemma \ref{strichartz_estimates}.
\begin{lemma}\label{lem:elementary>0}
For $j \in \Z$ and $c \in \R$ with $j \ge 0$ and $c > 0$, let
\begin{equation}\label{eq:h_j(x)}
h_j(x) = x^{2j+1} + (c-x)^{2j+1}.
\end{equation}
Then $h_j$ satisfies
\begin{enumerate}
\item $h_j$ is a symmetry about $x = \frac{c}{2}$.
\item $h_j(x) > 0$ for all $x \in \R$.
\item $h_j'(\frac{c}{2}) = 0$.
\end{enumerate}
\vspace{0.1cm}
If $j \ge 1$,
\begin{enumerate}
\item[(4)] $h_j''(x) > 0$ for all $x \in \R$.
\item[(5)] $h_j$ has only one absolute minimum value at $x = \frac{c}{2}$.
\item[(6)] $h_j$ can be written as
\begin{equation}\label{eq:h_j}
\begin{aligned}
h_j(x) =&~{} \left(x-\frac{c}{2} + \alpha\right)\left(x-\frac{c}{2} - \alpha\right)\sum_{n=1}^j\binom{2j+1}{2n}h_{j-n}\left(\frac{c}{2}\right)\left(\sum_{\ell =0}^{n-1}\left(x-\frac{c}{2} \right)^{2n-2-2\ell}\alpha^{2\ell}\right) \\
&~{}+ \sum_{n=0}^{j}\binom{2j+1}{2n}h_{j-n}\left(\frac{c}{2}\right)\alpha^{2n},
\end{aligned}
\end{equation}
for any $\alpha \in \R$.
\end{enumerate}
\end{lemma}

\begin{proof}
It is easy to see that
\[h_j\left(\frac{c}{2} + x\right) = \left(\frac{c}{2}+x\right)^{2j+1} + \left(\frac{c}{2}-x\right)^{2j+1} = h_j\left(\frac{c}{2} - x\right),\]
which satisfies Item (1).

\vspace{0.2cm}

For Item (2), it suffices to show $h_j(x) > 0$ for all $x \ge \frac{c}{2}$ thanks to Item (1). Obviously, $h_j(\frac{c}{2}) = \frac{c^{2j+1}}{2^{2j}} > 0$. Since
\[a^{2j+1} + b^{2j+1} = (a+b)\left(\sum_{n=0}^{2j}(-1)^n a^{2j-n}b^n\right) = (a+b)\left(\sum_{m=0}^{j-1}(a-b) a^{2j-2m-1}b^{2m} + b^{2j}\right),\]
for $x > \frac{c}{2}$, we have
\[h_j(x) = c\left(\sum_{m=0}^{j-1}(2x-c) x^{2j-2m-1}(c-x)^{2m} + (c-x)^{2j}\right).\]
Note that when $x = c$, we have $h_j(c) = c^{2j+1} > 0$. Thus, for all $x > \frac{c}{2}$ with $x \neq c$, all terms are strictly positive, which proves Item (2).

\vspace{0.2cm}

A direct computation gives
\[h_j'(x) = (2j+1)\left(x^{2j} - (c-x)^{2j}\right),\]
thus $h'(\frac{c}{2}) = 0$. This proves Item (3).

\vspace{0.2cm}

In what follows, we fix $j \ge 1$. Item (4) follows immediately from Item (2) due to
\[h_j''(x) = (2j+1)(2j)\left(x^{2j-1} + (c-x)^{2j-1}\right).\]

\vspace{0.2cm}

Item (5) immediately follows from Items (3) and (4).

\vspace{0.2cm}

For Item (6), we first show
\begin{equation}\label{eq:h_j-1}
h_j(x) = \sum_{n=0}^j\binom{2j+1}{2n}h_{j-n}\left(\frac{c}{2}\right)\left(x-\frac{c}{2}\right)^{2n}.
\end{equation}
When $j=1$, we see that
\[h_1(x) = x^3 + (c-x)^3 = 3c(x^2 - cx) + c^3 = 3c\left(x-\frac{c}{2}\right)^2 + \frac{c^3}{4}.\]
Since
\[\binom{3}{0}h_{1}\left(\frac{c}{2}\right) = 2 \cdot \left(\frac{c}{2}\right)^3 = \frac{c^3}{4} \quad \mbox{and} \quad \binom{3}{2}h_{0}\left(\frac{c}{2}\right) = 3c,\]
\eqref{eq:h_j-1} is true for $j=1$. Assume that \eqref{eq:h_j-1} is true for $j = m-1$. Then, by the induction hypothesis, we have
\[h_m''(x) = (2m+1)(2m)h_{m-1}(x) = (2m+1)(2m)\sum_{n=0}^{m-1}\binom{2m-1}{2n}h_{m-1-n}\left(\frac{c}{2}\right)\left(x-\frac{c}{2}\right)^{2n}.\]
Since $h_m'(\frac{c}{2}) = 0$, by integrating from $\frac{c}{2}$ to $x$ twice, we obtain
\[\begin{aligned}
h_m(x) =&~{} \sum_{n=0}^{m-1}\frac{(2m+1)(2m)}{(2n+2)(2n+1)}\binom{2m-1}{2n}h_{m-1-n}\left(\frac{c}{2}\right)\left(x-\frac{c}{2}\right)^{2n+2} + h_m\left(\frac{c}{2}\right)\\
=&~{} \sum_{n=0}^{m}\binom{2m+1}{2n}h_{m-n}\left(\frac{c}{2}\right)\left(x-\frac{c}{2}\right)^{2n}.
\end{aligned}\]
Thus, by the mathematical induction, we prove \eqref{eq:h_j-1} for all $j \ge 1$.

\vspace{0.2cm}

In order to derive \eqref{eq:h_j} from \eqref{eq:h_j-1}, it suffices to show that for each $n\in \mathbb{N}$
\begin{equation}\label{eq:induction}
\left(x-\frac{c}{2} \right)^{2n} = \left(x-\frac{c}{2} + \alpha\right)\left(x-\frac{c}{2} - \alpha\right)\sum_{\ell =0}^{n-1}\left(x-\frac{c}{2} \right)^{2n-2-2\ell}\alpha^{2\ell} + \alpha^{2n}.
\end{equation}
Indeed, if \eqref{eq:induction} is true, then we reduce \eqref{eq:h_j-1} as
\[\begin{aligned}
h_j(x) =&~{} \sum_{n=0}^j\binom{2j+1}{2n}h_{j-n}\left(\frac{c}{2}\right)\left(x-\frac{c}{2}\right)^{2n}\\
=&~{} \sum_{n=1}^j\binom{2j+1}{2n}h_{j-n}\left(\frac{c}{2}\right)\left(\left(x-\frac{c}{2} + \alpha\right)\left(x-\frac{c}{2} - \alpha\right)\sum_{\ell =0}^{n-1}\left(x-\frac{c}{2} \right)^{2n-2-2\ell}\alpha^{2\ell} + \alpha^{2n}\right) \\
&~{}+ \binom{2j+1}{0}h_{j}\left(\frac{c}{2}\right)\\
=&~{} \left(x-\frac{c}{2} + \alpha\right)\left(x-\frac{c}{2} - \alpha\right)\sum_{n=1}^j\binom{2j+1}{2n}h_{j-n}\left(\frac{c}{2}\right)\left(\sum_{\ell =0}^{n-1}\left(x-\frac{c}{2} \right)^{2n-2-2\ell}\alpha^{2\ell}\right) \\
&~{}+\sum_{n=1}^j\binom{2j+1}{2n}h_{j-n}\left(\frac{c}{2}\right)\alpha^{2n}+ \binom{2j+1}{0}h_{j}\left(\frac{c}{2}\right),
\end{aligned}\]
which proves \eqref{eq:h_j}.

We use, again,  the mathematical induction to prove \eqref{eq:induction}. When $n=1$, we  obtain
\begin{equation}\label{eq:simple}
\left(x-\frac{c}{2}\right)^2 = \left(x-\frac{c}{2}-\alpha\right)\left(x-\frac{c}{2}+\alpha\right) + \alpha^2.
\end{equation}
Assume that \eqref{eq:induction} is true for $n=m-1$. Then, using \eqref{eq:simple} and the induction hypothesis, we have
\[\begin{aligned}
\left(x-\frac{c}{2}\right)^{2m} =&~{} \left(\left(x-\frac{c}{2}-\alpha\right)\left(x-\frac{c}{2}+\alpha\right)+\alpha^2\right)\left(x-\frac{c}{2}\right)^{2m-2} \\
=&~{} \left(x-\frac{c}{2}-\alpha\right)\left(x-\frac{c}{2}+\alpha\right)\left(x-\frac{c}{2}\right)^{2m-2} \\
&~{}+ \alpha^2\left( \left(x-\frac{c}{2} + \alpha\right)\left(x-\frac{c}{2} - \alpha\right)\sum_{\ell =0}^{m-2}\left(x-\frac{c}{2} \right)^{2m-4-2\ell}\alpha^{2\ell} + \alpha^{2m-2}\right)\\
=&~{} \left(x-\frac{c}{2}-\alpha\right)\left(x-\frac{c}{2}+\alpha\right)\sum_{\ell =0}^{m-1}\left(x-\frac{c}{2} \right)^{2m-2-2\ell}\alpha^{2\ell} + \alpha^{2m},
\end{aligned}\]
which proves \eqref{eq:induction}.
\end{proof}

\begin{remark}
Collecting all information in Lemma \ref{lem:elementary>0}, one can roughly sketch the shape of $h_j(x)$ defined by \eqref{eq:h_j(x)} as in Figure \ref{Fig:0}:
\end{remark}
\begin{figure}[h!]
\begin{center}
\begin{tikzpicture}[scale=0.95]
\draw[->] (-0.5,0) -- (5,0) node[below] {$x$};
\draw[->] (0,-0.5) -- (0,5) node[left] {$h_j(x)$};
\draw[scale=0.5, domain=0:6, smooth, variable=\x] plot ({\x}, {(\x-3)*(\x-3) + 1.5});
\node at (1.5,0){$\bullet$};
\node at (0,0.75){$\bullet$};
\node at (1.5,-0.5){$\frac{c}{2}$};
\node at (-0.7,0.75){$h_j(\frac{c}{2})$};
\draw[thick,dashed] (1.5,0) -- (1.5,0.75);
\draw[thick,dashed] (0,0.75)--(1.5,0.75);
\end{tikzpicture}
\end{center}
\caption{Lemma \ref{lem:elementary>0} describes that $h_j(x)$ has a convex and symmetric form, and its slope is strictly increasing if $x > \frac{c}{2}$, while strictly decreasing otherwise.} \label{Fig:0}
\end{figure}
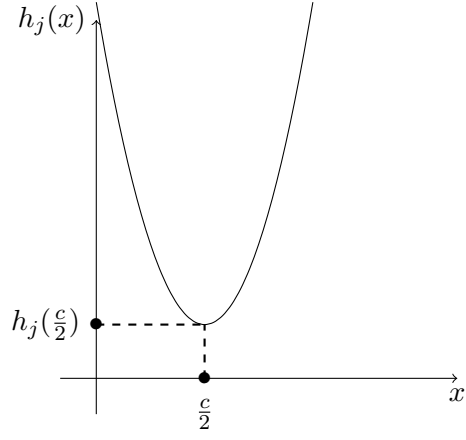

\begin{proof}[Proof of Lemma \ref{strichartz_estimates}]
The proof basically follows the proof of Lemma 2.1 in \cite{TT2004} associated with the Airy flow. Moreover, we refer to \cite{Kwak2019} for the case when $j=2$. Thus, in the proof below, we fix $j \ge 3$.

Let $f = f_1 + f_2$, where
\[\wh{f}_1(k) = 0, \quad \mbox{if} \quad |k| > 1.\]
Note that $|\{k \in \Z : k \in \supp (\wh{f}_1)\}| = 3$. Since $f^2 \le 2f_1^2 + 2f_2^2$, it suffices to treat $\norm{f_1^2}_{L^2(\R \times \T)}$ and $\norm{f_2^2}_{L^2(\R \times \T)}$ separately.

\vspace{0.2cm}

\textbf{$f_1^2$ case.} A computation gives
\begin{equation}\label{eq:L^4_1}
\norm{f_1^2}_{L^2(\R \times \T )}^2 \le \sum_{k\in\Z} \int_{\R} \left| \sum_{k_1\in\Z}\int_\R  |\wt{f}_1(\tau_1,k_1)||\wt{f}_1(\tau -\tau_1,k-k_1)|\; d\tau_1 \right|^2 \; d\tau .
\end{equation}
From the support property, the right-hand side of \eqref{eq:L^4_1} vanishes unless $|k| \le 2j$. Let
\[\wt{F}_1(\tau, k) = \bra{\tau- k^{2j+1}}^{b}|\wt{f}_1(\tau,k)|.\]
The Cauchy-Schwarz inequality and the Minkowski inequality, we see that for $b > \frac14$,
\[\begin{aligned}
\mbox{RHS of } \eqref{eq:L^4_1} \lesssim&~{}\sum_{\substack{k \in \Z  \\ |k| \le 2}} \int_{\R} \Bigg( \sum_{k_1 \in \Z}\left(\int_{\R} \bra{\tau-\tau_1 - (k-k_1)^{2j+1}}^{-2b}\bra{\tau_1- k_1^{2j+1}}^{-2b} \; d\tau_1 \right)^{\frac12}\\
& \hspace{7em} \times \left(\int_\R  |\wt{F}_1(\tau_1,k_1)|^2|\wt{F}_1(\tau -\tau_1,k-k_1)|^2\; d\tau_1\right)^{\frac12}  \Bigg)^2 \; d\tau \\
\lesssim&~{} \sum_{\substack{k \in \Z  \\ |k| \le 2}} \Bigg(\sum_{k_1 \in \Z }\left(\int_{\R^2}  |\wt{F}_1(\tau_1,k_1)|^2|\wt{F}_1(\tau -\tau_1,k-k_1)|^2\; d\tau_1 d \tau \right)^{\frac12}\Bigg)^2\\
\lesssim&~{} \norm{f_1}_{X^{0,b}}^4 \lesssim \norm{f}_{X^{0,b}}^4.
\end{aligned}\]

\vspace{0.2cm}

\textbf{$f_2^2$ case.} Analogous to \eqref{eq:L^4_1}, we have
\[
\begin{aligned}
\norm{f_2^2}_{L^2(\R \times \T )}^2 \le&~{} \sum_{k \in \Z } \int_{\R} \left| \sum_{k_1 \in \Z }\int_\R  |\wt{f}_2(\tau_1,k_1)||\wt{f}_2(\tau -\tau_1,k-k_1)|\; d\tau_1  \right|^2 \; d\tau\\
=&~\sum_{\substack{k \in \Z  \\ |k| \le 1}} \int_{\R} \left| \sum_{k_1 \in \Z }\int_\R  |\wt{f}_2(\tau_1,k_1)||\wt{f}_2(\tau -\tau_1,k-k_1)|\; d\tau_1  \right|^2 \; d\tau \\
&+\sum_{\substack{k \in \Z  \\ |k| >1}} \int_{\R} \left| \sum_{k_1 \in \Z }\int_\R  |\wt{f}_2(\tau_1,k_1)||\wt{f}_2(\tau -\tau_1,k-k_1)|\; d\tau_1 \right|^2 \; d\tau \\
=:&~{} I_1 + I_2.
\end{aligned}\]
The term $I_1$ can be treated similarly as \textbf{$f_1^2$ case}. For the term $I_2$, we may assume that $k_1 > 1$ and $k-k_1>1$ (thus, $k > 1$). Indeed, let $f_2 = f_{2,1} + f_{2,2}$, where
\[\wh{f}_{2,1}(k) = 0 \quad \mbox{if} \quad k > 1,\]
then $\norm{f_2^2}_{L^2}^2 \le 2\norm{f_{2,1}^2}_{L^2}^2 + 2\norm{f_{2,2}^2}_{L^2}^2$ and $\norm{f_{2,1}^2}_{L^2} = \norm{\overline{f_{2,1}}^2}_{L^2} = \norm{f_{2,2}^2}_{L^2}$. Similarly as \eqref{eq:L^4_1}, we have
\[\begin{aligned}
I_2 \lesssim&~{}\sum_{\substack{k \in \Z  \\ k >1}} \int_{\R} \Bigg( \Bigg(\sum_{\substack{k_1 \in \Z  \\ k_1, k-k_1 > 1}}\int_\R \bra{\tau-\tau_1 -(k-k_1)^{2j+1}}^{-2b}\bra{\tau_1- k_1^{2j+1}}^{-2b} \; d\tau_1 \Bigg)^{\frac12}\\
& \hspace{7em} \times\Bigg(\sum_{k_1 \in \Z }\int_\R  |\wt{F}_2(\tau_1,k_1)|^2|\wt{F}_2(\tau -\tau_1,k-k_1)|^2\; d\tau_1\Bigg)^{\frac12}  \Bigg)^2 \; d\tau \\
\lesssim&~{} M\norm{f_2}_{X^{0,b}}^4,
\end{aligned}\]
where
\[M = \sup_{\substack{\tau \in \R, k \in \Z  \\ k >1}}\sum_{\substack{k_1 \in \Z  \\ k_1, k-k_1 > 1}}\int_\R \bra{\tau-\tau_1 - (k-k_1)^{2j+1}}^{-2b}\bra{\tau_1- k_1^{2j+1}}^{-2b} \; d\tau_1.\]
Thus, it is enough to show that $M \lesssim 1$ whenever $b > \frac{j+1}{2(2j+1)}$.

A direct computation
\[\int_{\R} \bra{a}^{-\alpha} \bra{b-a}^{-\alpha} \; da \lesssim \bra{b}^{1-2\alpha},\]
for $\frac12 < \alpha < 1$, yields
\[M \lesssim \sup_{\substack{\tau \in \R, k \in \Z  \\ k > 1}}  \sum_{\substack{k_1 \in \Z  \\ k_1, k-k_1 > 1}} \bra{\tau-k_1^{2j+1} - (k-k_1)^{2j+1}}^{1-4b}.\]

For each $\tau \in \R$ and $k \in \Z$ with $k > 1$, let $h(x) := h_j(x) - \tau$, for $h_j$ as in \eqref{eq:h_j(x)} with $c = k$. From \eqref{eq:h_j-1}, we know
\[h(x) = \sum_{n=1}^j\binom{2j+1}{2n}h_{j-n}\left(\frac{k}{2}\right)\left(x-\frac{k}{2}\right)^{2n} + h_j\left(\frac{k}{2}\right) - \tau.\]
From Lemma \ref{lem:elementary>0} Item (5), we know $h_j\left(\frac{k}{2}\right) - \tau$ is the absolute minimum value of $h$. If $h_j\left(\frac{k}{2}\right) - \tau \ge 0$, we know
\[\langle h(k_1) \rangle^{1-4b} \le \left((2j+1)h_{0}\left(\frac{k}{2}\right)\left(k_1-\frac{k}{2}\right)^{2j}\right)^{1-4b} =\left( (2j+1)k\left(k_1-\frac{k}{2}\right)^{2j}\right)^{1-4b} \quad \mbox{on} \;\; A^c,\]
where the set
\begin{equation}\label{eq:setA}
A = \left\{k_1 \in \Z : \left|k_1 - \frac{k}{2}\right| \le 1\right\}.
\end{equation}
Note that $|A| \le 3$. Thus,
\[\begin{aligned}
\sum_{\substack{k_1 \in \Z  \\ k_1, k-k_1 > 1}} \bra{\tau-k_1^{2j+1} - (k-k_1)^{2j+1}}^{1-4b} \lesssim&~{} \sum_{k_1 \in A} 1 + \sum_{\substack{k_1 \in A^c \\ k_1, k-k_1 > 1}}\left(k\left(k_1-\frac{k}{2}\right)^{2j}\right)^{1-4b}\\
\lesssim&~{} 1 + \sum_{k > 1}k^{(2j+1)(1-4b)} + \sum_{|k_1-\frac{k}{2}| > 1}\left|k_1-\frac{k}{2}\right|^{(2j+1)(1-4b)}\\
\lesssim&~{} 1,
\end{aligned}\]
provided that $b > \frac{j+1}{2(2j+1)}$.

On the other hand, if $h_j\left(\frac{k}{2}\right) - \tau < 0$, since $h$ is symmetry about $x = \frac{k}{2}$ and has the absolute minimum value at $\frac{k}{2}$, there is $\alpha > 0$ such that $\frac{k}{2} + \alpha$ and $\frac{k}{2} - \alpha$ are the only roots of $h$, i.e., $h(\frac{k}{2} + \alpha) = h(\frac{k}{2} - \alpha) = 0$. A direct computation gives
\[0 = h\left(\frac{k}{2} + \alpha\right) = h\left(\frac{k}{2} - \alpha\right) = \sum_{n=0}^j\binom{2j+1}{2n}h_{j-n}\left(\frac{k}{2}\right)\alpha^{2n} - \tau.\]
With this, by Lemma \ref{lem:elementary>0} Item (6), we know
\[h(x) = \left(x-\frac{k}{2} + \alpha\right)\left(x-\frac{k}{2} - \alpha\right)\sum_{n=1}^j\binom{2j+1}{2n}h_{j-n}\left(\frac{k}{2}\right)\left(\sum_{\ell =0}^{n-1}\left(x-\frac{k}{2} \right)^{2n-2-2\ell}\alpha^{2\ell}\right).\]
Let $\Omega_{\pm}$ be a set of $k_1$ defined by
\[\Omega_{\pm} = \left\{k_1 \in \Z : \left|k_1 - \frac{k}{2} \pm \alpha\right| \right\}.\]
Note that $|\Omega_{\pm}| \le 3$. Then, similarly as before, we have
\[\begin{aligned}
\langle h(k_1) \rangle^{1-4b} \le&~{} \left((2j+1)h_{0}\left(\frac{k}{2}\right)\left(k_1-\frac{k}{2} + \alpha\right)\left(k_1-\frac{k}{2} - \alpha\right)\left(k_1-\frac{k}{2}\right)^{2j-2}\right)^{1-4b} \\
=&~{}\left( (2j+1)k\left|k_1-\frac{k}{2} + \alpha\right|\left|k_1-\frac{k}{2} - \alpha\right|\left(k_1-\frac{k}{2}\right)^{2j-2}\right)^{1-4b}
\end{aligned}\]
on $ (\Omega_+ \cup \Omega_- \cup A)^c$, where the set $A$ is as in \eqref{eq:setA}. Therefore, we conclude that
\[\begin{aligned}
&~{}\sum_{\substack{k_1 \in \Z  \\ k_1, k-k_1 > 1}} \bra{\tau-k_1^{2j+1} - (k-k_1)^{2j+1}}^{1-4b}\\
 \lesssim&~{} \sum_{k > 1}k^{(2j+1)(1-4b)} + \sum_{|k_1-\frac{k}{2} + \alpha| > 1}\left|k_1-\frac{k}{2} + \alpha\right|^{(2j+1)(1-4b)}\\
&~{}+ \sum_{|k_1-\frac{k}{2} - \alpha| > 1}\left|k_1-\frac{k}{2} - \alpha\right|^{(2j+1)(1-4b)} + \sum_{|k_1-\frac{k}{2}|>1}\left|k_1-\frac{k}{2}\right|^{(2j+1)(1-4b)}\\
\lesssim&~{} 1,
\end{aligned}\]
provided that $b > \frac{j+1}{2(2j+1)}$.
This completes the proof.
\end{proof}

\section{Propagation of singularities and unique continuation property}\label{Apendice}
In this section, we provide necessary basic tools that we use to demonstrate the main results of this work. 

\subsection{Auxiliary lemmas}We first recall three useful lemmas for our analyses below.
\begin{lemma}\cite[Lemma A.1]{PV}\label{lem:weight}
A function $\phi \in C^{\infty}(\T)$ can be written in the form $\partial_x \varphi$ for some function $\varphi \in C^{\infty}(\T)$ if and only if
$\int_\T \phi(x) \; dx = 0$.
\end{lemma}

\begin{lemma}\cite[Lemma A.1.]{Laurent-esaim}\label{lem:La1}
Let $s, r \in \R$. Let $f$ denote the operator of multiplication by $f \in C^{\infty}(\T)$. Then, $[D^r,f]$ maps any $H^s$ into $H^{s-r+1}$, where $D^r$ operator is defined on distributions $\mathcal D'(\T)$ by
\[\widehat{D^r f}(n) = \begin{cases} \mbox{sgn}(n) |n|^r \hat{f}(n), \quad &\mbox{if} \;\; n \neq 0, \\  \hat{f}(0), \quad &\mbox{if} \;\; n = 0.\end{cases}\]
\end{lemma}

\begin{lemma}\cite[Lemma A.3.]{Laurent-esaim}\label{lem:La2}
Let $s \in\R$. Let $f \in C^{\infty}(\T)$ and $\rho_{\vep} = e^{\vep^2\partial_x^2}$ with $0 \le \vep \le 1$. Then, $[\rho_{\vep}, f]$ is uniformly bounded as an operator from $H^s$ into $H^{s+1}$.
\end{lemma}

We end this subsection with the multiplication property of $X^{s,b}$ spaces.

\begin{lemma}[Multiplication property]\label{estimates_a}
Let $-1\leq b\leq1$, $s\in\mathbb{R}$ and $\varphi\in C^{\infty}(\T)$. Then, $\varphi u \in X^{s-2j|b|, b}$,  for any $u\in X^{s,b}$. Moreover, the map $u \mapsto \varphi u$ from $X_T^{s,b}$ into $X_T^{s-2j|b|, b}$ is bounded.
\end{lemma}

\begin{proof}
Since $X^{s,b}$ space is stable with respect to time localization (see \eqref{eq:Xsb stable}), it is enough to prove the first part (without time localization). When $b=0$, it is obvious thanks to $X^{s,0} = L^2(\R;H^s)$ (see \cite[Theorem 4.3]{PV}) and
\begin{equation}\label{eq:smooth}
\norm{\varphi u}_{H^s} \lesssim \norm{u}_{H^s},
\end{equation}
where the implicit constant depends on $s$ and $\varphi$.

We now take $b=1$, then it is known that
\begin{equation}\label{eq:Xs1}u \in X^{s,1} \quad \Longleftrightarrow \quad u \in L^2(\R;H^s) \quad \mbox{and} \quad \left(\partial_t + (-1)^{j+1}\partial_x^{2j+1}\right)u \in L^2(\R; H^s),
\end{equation}
thanks to the definition of $X^{s,b}$ and $\bra{\cdot} \sim 1 + |\cdot|$. A computation gives
\begin{equation}\label{eq:commutator}
\left(\partial_t + (-1)^{j+1}\partial_x^{2j+1}\right)(\varphi u) = \varphi \left(\partial_t + (-1)^{j+1}\partial_x^{2j+1}\right) u - [\varphi, (-1)^{j+1}\partial_x^{2j+1}]u,
\end{equation}
where $[\cdot, \cdot]$ is the standard commutator operator defined by $[A,B] = AB - BA$. Thanks to \eqref{eq:commutator} and \eqref{eq:smooth} in the definition \eqref{eq:Xs1}, it suffices to show
\begin{equation}\label{eq:b1}
\norm{[\varphi, (-1)^{j+1}\partial_x^{2j+1}]u}_{H^{s-2j}} \lesssim \norm{u}_{H^s}.
\end{equation}
Observe that
\[[\varphi, (-1)^{j+1}\partial_x^{2j+1}]u = (-1)^j \sum_{\ell = 0}^{2j} \binom{2j+1}{\ell}\partial_x^{2j+1-\ell} \varphi~ \partial_x^{\ell}u.\]
This, in addition to \eqref{eq:smooth}, immediately implies \eqref{eq:b1}. Thus, by the complex interpolation theorem of Stein-Weiss for weighted $L^{p}$ spaces (see \cite[p. 114]{Berg}), we complete the proof for the case when $0 \le b \le 1$.

The case when $ -1 \le b \le 0$ can be proved via the duality argument. Precisely, the duality argument ensures the map $u \mapsto \varphi u$ from $X^{-s+2jb,-b}$ to $X^{-s-b}$ is bounded for $0 \le b \le 1$. Since the spatial regularity is arbitrary, by replacing $-s+2jb$ by $s$, we conclude that multiplication map is bounded from $X^{s,-b}$ to $X^{s-2jb, -b}$, which implies the desired result for $-1 \le b \le 0$, we thus complete the proof.
\end{proof}

\subsection{Propagation of compactness}
In this section, we present the properties of propagation of compactness for the linear differential operator $L=\partial_{t}+ (-1)^{j+1} \px^{2j+1}$ associated with the higher-order KdV type equation. The main ingredient is basically pseudo-differential analysis.

\begin{proposition}\label{prop_a} Let $T>0$ and $0\leq b^{\prime}\leq b\leq1$ be given, with $b>0$. Suppose that $u_{n}\in X_T^{0,b}$ and $f_{n}\in X_T^{-2j + 2jb, -b}$ satisfy
\[
\partial_{t}u_{n}+(-1)^{j+1} \px^{2j+1}u_{n}=f_{n},
\]
for $n \in \N$.  Assume that there exists a constant $C>0$ such that%
\begin{equation}
\left\Vert u_{n}\right\Vert _{X_T^{0,b}}\leq C \label{k14}%
\end{equation}
and that
\begin{equation}
\left\Vert u_{n}\right\Vert _{X_T^{-2j+2jb,-b}}+\left\Vert f_{n}\right\Vert
_{X_T^{-2j+2jb,-b}}+\left\Vert u_{n}\right\Vert _{X_T^{-1+2jb',-b'}}\rightarrow0, \;\; \text{ as }n\rightarrow+\infty\text{.}
\label{k15}%
\end{equation}
In addition, assume that for some nonempty open set $\omega\subset\mathbb{T}$
it holds%
\begin{equation}\label{eq:assumption}
u_{n}\rightarrow0\text{ strongly in }L^2\left(  0,T;L^2\left(
\omega\right)  \right)  \text{.}%
\end{equation}
Then, 
\[
u_{n}\rightarrow0\text{ strongly in }L^2_{loc}\left(  \left(  0,T\right)
;L^2\left(  \mathbb{T}\right)  \right),\;\;\text{as}\;n\rightarrow +\infty.%
\]
\end{proposition}

\begin{proof}
For any compact interval $I \subset (0,T)$, we choose a cut-off  function $\psi \in C_c^{\infty}((0,T))$ such that $0 \le \psi \le 1$ and $\psi \equiv 1$ in $I$. Then, a simple computation yields
\[\norm{u_n}_{L^2(I;L^2)}^2 \le \int_0^T \psi(t) (u_n, u_n) \; dt.\]
On the other hand, since $\T$ is compact, there exist a finite number of open interval of the length less than the size of $\omega$ centered at $x_0^m$, $m=1,2,\cdots,M$ for some $M$. For an appropriate $\chi \in C_c^{\infty}(\omega)$, we can construct a partition of unity as
\begin{equation}\label{eq:PoU}
0 \le \chi(x-x_0^m) \le 1 \quad \mbox{and} \quad \sum_{m=1}^M \chi(x-x_0^m) \equiv 1,
\end{equation}
for all $x \in \T$ and $m=1, \cdots,M$. Then,
\[\int_0^T \psi(t) (u_n, u_n) \; dt \le \sum_{m=1}^M \int_0^T \psi(t) (\chi(\cdot - x_0^m) u_n, u_n) \; dt,\]
which reduces our problem to proving that for any $\chi \in C_c^{\infty}(\omega)$ and $x_0 \in \T$
\begin{equation}\label{eq:con1}
\left(\psi(t)\chi(\cdot-x_0) u_{n}, u_{n}\right)_{L^2(0,T;L^2)} \rightarrow 0, \quad n \to \infty.
\end{equation}
Once proving that 
\begin{equation}\label{eq:con2}
\lim_{n \to \infty}\left(\psi(t) (\partial_x \varphi) u_{n}, u_{n}\right)_{L^2(0,T;L^2)} =0,
\end{equation}
for some $\varphi \in C^{\infty}(\T)$, we immediately obtain \eqref{eq:con1} by putting
\begin{equation}\label{eq:weight}
\partial_x\varphi = \chi(x) - \chi(x-x_0).
\end{equation}
Indeed, a direct computation gives
\[\left(\psi(t)\chi(x-x_0) u_{n}, u_{n}\right)_{L^2(0,T;L^2)} = \left(\psi(t)\chi(x) u_{n}, u_{n}\right)_{L^2(0,T;L^2)} - \left(\psi(t) (\partial_x \varphi) u_{n}, u_{n}\right)_{L^2(0,T;L^2)},\]
and the right-hand side goes to zero thanks to \eqref{eq:assumption} and \eqref{eq:con2}, which implies \eqref{eq:con1}. Note that Lemma \ref{lem:weight} ensures to find $\varphi \in C^{\infty}(\mathbb{T})$ satisfying \eqref{eq:weight}. Thus, we are now further reduced to proving \eqref{eq:con2}.

On the other hand, one knows from Plancherel's theorem that
\[\begin{aligned}
\int_\T \partial_x \varphi (x) u_{n}(x) u_{n}(x) \; dx 
=&~{} \widehat{u_{n}}(0)\int_\T \partial_x \varphi (x) u_{n} (x) \; dx\\
&+ \int_\T (-1)^j\partial_x^{2j}D^{-2j}u_{n}(x) \partial_x \varphi (x) u_{n} (x) \; dx.
\end{aligned}\]
Therefore, the proof of Proposition \ref{prop_a} is completed from
\begin{equation}\label{eq:step1}
\lim_{n \to \infty}\left|\left(\psi(t) (\partial_x \varphi) \partial_x^{2j}D^{-2j} u_{n}, {u_n}\right)_{L^2(0,T;L^2)}\right| = 0
\end{equation}
and
\begin{equation}\label{eq:step2}
\lim_{n \to \infty}\left|\left(\psi(t) (\partial_x \varphi) \widehat{u_{n}}(t,0), u_{n}\right)_{L^2(0,T;L^2)}\right| = 0.
\end{equation}

\subsubsection*{\textbf{Proof of \eqref{eq:step1}}} Let take real valued $\varphi \in C^{\infty}(\T)$ (satisfying \eqref{eq:weight}) and $\psi \in C_0^{\infty}((0,T))$, and let set
\[\mathcal B:=\varphi (x) D^{-2j} \quad \mbox{and} \quad \mathcal A:= \psi(t) \mathcal B.\]
It is straightforward to know $\mathcal A^*= \psi(t) D^{-2j} \varphi(x)$. We denote by $\mathcal A_{\vep}$ the regularization of $\mathcal A$ by
\[\mathcal A_{\vep} := \mathcal A e^{\vep \partial_x^2} =: \psi(t) \mathcal B_{\vep},\]
and set $\alpha_{n,\vep}= \left([\mathcal A_{\vep}, \mathcal L]u_n, u_n\right)_{L^2(0,T;L^2)}$, where $\mathcal L = \pt + (-1)^{j+1}\px^{2j+1}$. From $\mathcal L u_n = f_n$ and $\mathcal L^* = -\mathcal L$, one has
\[\alpha_{n, \vep} = \left(f_n, \mathcal A_{\vep}^*u_n\right)_{L^2(0,T;L^2)} + \left(\mathcal A_{\vep} u_n, f_n\right)_{L^2(0,T;L^2)}.\]
Using \eqref{eq:Xsb stable} and Lemma \ref{estimates_a}, the Cauchy-Schwarz inequality yields
\[\left|\left(f_n, \mathcal A_{\vep}^*u_n\right)_{L^2(0,T;L^2)}\right| \le \norm{f_n}_{X_T^{-2j+2jb, -b}}\norm{\mathcal A_{\vep}^*u_n}_{X_T^{2j-2jb,b}} \lesssim \norm{f_n}_{X_T^{-2j+2jb, -b}}\norm{u_n}_{X_T^{0,b}}.\]
Similarly, we show
\[\left|\left(\mathcal A_{\vep} u_n, f_n\right)_{L^2(0,T;L^2)}\right| \lesssim \norm{u_n}_{X_T^{0,b}}\norm{f_n}_{X_T^{-2j+2jb, -b}}.\]
Thus, the assumption \eqref{k15} ensures
\[\lim_{n \to \infty} \sup_{0<\vep \le 1} \alpha_{n,\vep} = 0.\]
On the other hand, the fact $\partial_t \mathcal A_{\vep} = \psi'(t)\mathcal B_{\vep} + \mathcal A_{\vep} \partial_t$ enables us to rewrite
\[\alpha_{n,\vep}=\left([\mathcal A_{\vep}, (-1)^{j+1}\partial_x^{2j+1}]u_n, u_n\right)_{L^2(0,T;L^2)} - \left(\psi'(t)\mathcal B_{\vep} u_n, u_n\right)_{L^2(0,T;L^2)}.\]
Then, the analogous argument shows
\[ \left|\left(\psi'(t)\mathcal B_{\vep} u_n, u_n\right)_{L^2(0,T;L^2)}\right| \lesssim \norm{u_n}_{X_T^{-2j+2jb, -b}}\norm{u_n}_{X_T^{0,b}},\]
which implies
\[\lim_{n \to \infty} \sup_{0<\vep \le 1}\left|\left(\psi'(t)\mathcal B_{\vep} u_n, u_n\right)_{L^2(0,T;L^2)}\right| = 0.\]
Hence, we conclude that
\[\lim_{n \to \infty} \sup_{0< \vep \le 1} \left([\mathcal A_{\vep}, (-1)^{j+1}\partial_x^{2j+1}]u_n, u_n\right)_{L^2(0,T;L^2)} = 0.\]
In particular,
\[\lim_{n \to \infty} \left([\mathcal A, (-1)^{j+1}\partial_x^{2j+1}]u_n, u_n\right)_{L^2(0,T;L^2)} = 0.\]
Since $\partial_x$ commutes with $D^{-1}$, a straightforward computation gives
\begin{equation}\label{eq:commutator1}
[\mathcal A, (-1)^{j+1}\partial_x^{2j+1}] = (-1)^j\psi(t)\sum_{\ell = 1}^{2j+1} \binom{2j+1}{\ell}\partial_x^{\ell} \varphi(x) \partial_x^{2j+1-\ell}D^{-2j}.
\end{equation}
Analogously, we show for $\ell =2, \cdots, 2j+1$ that\footnote{It suffices to choose $b' = b$ when $\ell = 2j+1$.}
\[\begin{aligned}
&~{}\left|\left((-1)^j\psi(t) \binom{2j+1}{\ell}\partial_x^{\ell} \varphi(x) \partial_x^{2j+1-\ell}D^{-2j} u_n, u_n\right)_{L^2(0,T;L^2)}\right|\\
\le&~{} \norm{\psi(t) \binom{2j+1}{\ell}\partial_x^{\ell} \varphi(x) \partial_x^{2j+1-\ell}D^{-2j}u_n}_{X_T^{0, -b'}}\norm{u_n}_{X_T^{0,b'}}\\
\lesssim&~{} \norm{u_n}_{X_T^{1-\ell+2jb', -b'}}\norm{u_n}_{X_T^{0,b'}},
\end{aligned}\]
which implies
\begin{equation}\label{eq:commutator2}
\lim_{n \to \infty}\left|\left((-1)^j\psi(t) \binom{2j+1}{\ell}\partial_x^{\ell} \varphi(x) \partial_x^{2j+1-\ell}D^{-2j} u_n, u_n\right)_{L^2(0,T;L^2)}\right| = 0.
\end{equation}
Collecting \eqref{eq:commutator1} and \eqref{eq:commutator2}, we complete the proof of \eqref{eq:step1}.

\subsubsection*{\textbf{Proof of \eqref{eq:step2}.}} A straightforward computation in addition to \eqref{k14} yields
\[\norm{\widehat{u_n}(t,0)}_{L^2((0,T))} \lesssim \norm{u_{n}}_{X_T^{0,b}} \le C.\]
Thus, the sequence $\widehat{u_n}(\cdot,0)$ is bounded in $H^{b}(0,T),$ which is compactly embedded in $L^{2}(0,T),$ by the Rellich Theorem. Therefore, there exists a subsequence that converges strongly in $L^{2}(0,T).$ Next, it can be seen that the only weak limit of a subsequence in $L^{2}(0,T)$ is zero, so that the whole sequence tends strongly to $0$ in $L^{2}(0,T).$ Hence,
$$\widehat{u_{n}}(t,0)\rightarrow 0,\;\text{(strongly) in}\;L^{2}(0,T),\;\text{as}\;n\rightarrow +\infty,$$  
and \eqref{eq:step2} holds. Consequently, Proposition \ref{prop_a} is proved.
\end{proof}

\subsection{Propagation of regularity} We now present the properties of propagation of regularity for the linear differential operator $\mathcal L=\partial_{t} + (-1)^{j+1}\partial_x^{2j+1}$ associated with the higher-order KdV-type equation.

\begin{proposition}\label{prop_b}
Let $T>0$, $r\in\mathbb{R}$, $0\leq b\leq1$ and $f\in X_T^{r,-b}$ be given. Let\ $u\in X_T^{r,b}$ be a solution of%
\[
\partial_{t}u+(-1)^{j+1} \px^{2j+1}u=f\text{.}%
\]
If there exists a nonempty $\omega\subset\mathbb{T}$ such that $u\in L_{loc}^{2}\left(  \left(  0,T\right)  ,H^{r+\rho}\left(  \omega\right)
\right)  $ for some $\rho$ with%
\[
0<\rho\leq\min\left\{  j(1-b),\frac{1}{2}\right\}  \text{,}%
\]
then%
\[
u\in L_{loc}^{2}\left(  \left(  0,T\right)  ,H^{r+\rho}\left(  \mathbb{T}%
\right)  \right)  \text{.}%
\]
\end{proposition}
\begin{proof}
The strategy of the proof is analogous to the proof of Proposition \ref{prop_a}. Let $s:= r + \rho$. For any compact interval $I \subset (0,T)$ and $\psi \in C_c^{\infty}((0,T))$ as in the proof of Proposition \ref{prop_a},  we have
\[\begin{aligned}
\norm{u}_{L^2(I;H^s)}^2 \le&~{} \int_0^T \psi(t) (u,u)_{H^s} \; dt\\
\lesssim&~{} \norm{u}_{L^2(0,T;L^2)}^2 + \left(\psi(t)D^{2s-2j}(-\partial_x^2)^{j}u,u\right)_{L^2(0,T;L^2)},
\end{aligned}\]
Hence, we are reduced to proving
\[\left|\left(\psi(t)D^{2s-2j}\partial_x^{2j}u,u\right)_{L^2(0,T;L^2)} \right| \lesssim 1.\]
On the other hand, using a partition of unity as in \eqref{eq:PoU} (but $\chi^2$ instead of $\chi$\footnote{It is possible if we simply take $\bar{\chi} = \sqrt{\chi}$.}), it is enough to show that for any $\chi \in C_c^{\infty}(\omega)$ and $x_0 \in \T$, we have
\[\left|\left(\psi(t)D^{2s-2j}\chi^2(x-x_0)\partial_x^{2j}u,u\right)_{L^2(0,T;L^2)} \right| \lesssim 1.\]
Moreover, by taking  $\partial_x\varphi = \chi^2(x) - \chi^2(x-x_0)$, we are finally reduced to proving
\begin{equation}\label{eq:proof3}
\left|\left(\psi(t)D^{2s-2j}(\partial_x \varphi)\partial_x^{2j}u,u\right)_{L^2(0,T;L^2)} \right| \lesssim 1,
\end{equation}
for some $\varphi \in C^{\infty}(\T)$, and
\begin{equation}\label{eq:proof4}
\left|\left(\psi(t)D^{2s-2j}\chi^2(x)\partial_x^{2j}u,u\right)_{L^2(0,T;L^2)} \right| \lesssim 1.
\end{equation}

\subsubsection*{\textbf{Proof of \eqref{eq:proof3}}}  For $n \in \N$, set
\[u_n:= e^{\frac1n \partial_x^2} u \quad \mbox{and} \quad f_n := e^{\frac1n \partial_x^2}f.\]
Note that $\mathcal L u_n = f_n$. Then, there exists $C>0$ such that
\begin{equation*}
\norm{u_n}_{X_T^{r,b}}, ~\norm{f_n}_{X_T^{r,-b}} \le C
\end{equation*}
for all $n \in\ N$. Define operators $\mathcal A$ and $\mathcal B$ by
\[\mathcal B:=D^{2s-2j}\varphi (x)  \quad \mbox{and} \quad \mathcal A:= \psi(t) \mathcal B,\]
then we know similarly as in the proof of Proposition \ref{prop_a} that
\[\begin{aligned}
\left([\mathcal A, (-1)^{j+1}\partial_x^{2j+1}]u_n, u_n\right)_{L^2(0,T;L^2)} - \left(\psi'(t)\mathcal B u_n, u_n\right)_{L^2(0,T;L^2)} =&~{}  \left(f_n, \mathcal A^*u_n\right)_{L^2(0,T;L^2)} \\&+ \left(\mathcal A u_n, f_n\right)_{L^2(0,T;L^2)}.
\end{aligned}\]
Using \eqref{eq:Xsb stable} and Lemma \ref{estimates_a}, one shows
\[\begin{aligned}
\left|\left(\mathcal A u_n, f_n\right)_{L^2(0,T;L^2)}\right| \le&~{}\norm{\mathcal A u_n}_{X_T^{-r,b}}\norm{f_n}_{X_T^{r,-b}}\\
 \lesssim&~{}\norm{u_n}_{X_T^{-r+2jb +2s -2j,b}}\norm{f_n}_{X_T^{r,-b}}\\
 \lesssim&~{}\norm{u_n}_{X_T^{r,b}}\norm{f_n}_{X_T^{r,-b}} \\
 \lesssim&~{} 1,
\end{aligned}\]
since $-r+2jb +2s -2j = r + 2\rho -2j(1-b) \le r$. Analogously, we show
\[\left|\left(\psi'(t)\mathcal B u_n, u_n\right)_{L^2(0,T;L^2)}\right|, ~\left| \left(f_n, \mathcal A^*u_n\right)_{L^2(0,T;L^2)}\right| \lesssim 1,\]
which says
\[\left|\left([\mathcal A, (-1)^{j+1}\partial_x^{2j+1}]u_n, u_n\right)_{L^2(0,T;L^2)}\right|\ \lesssim 1.\]
Note that the implicit constant, here, does not depend on $n \in \N$. Similarly as in \eqref{eq:commutator1}, we know
\[[\mathcal A, (-1)^{j+1}\partial_x^{2j+1}] = (-1)^j\psi(t)D^{2s-2j}\sum_{\ell = 1}^{2j+1} \binom{2j+1}{\ell}  \partial_x^{\ell}\varphi(x) \partial_x^{2j+1-\ell}.\]
For $\ell = 2, \cdots, 2j+1$, a direct computation gives
\[\begin{aligned}
&~{}\left|\left((-1)^j\psi(t)D^{2s-2j}\binom{2j+1}{\ell}  \partial_x^{\ell}\varphi(x) \partial_x^{2j+1-\ell} u_n, u_n \right)_{L^2(0,T;L^2)}\right|\\
 \lesssim&~{}\norm{\psi(t)D^{2s-2j}\partial_x^{\ell}\varphi(x) \partial_x^{2j+1-\ell} u_n}_{L^2(0,T;H^{-r})}\norm{u_n}_{L^2(0,T;H^r)}\\
\lesssim&~\norm{u_n}_{L^2(0,T;H^{-r+2s+1-\ell})}\norm{u_n}_{L^2(0,T;H^r)}.
\end{aligned}\]
Since $-r+2s+1-\ell = r +2\rho - 1 + (2 - \ell) \le r$, for all $\ell =2, \cdots 2j+1$, whenever $\rho \le \frac12$, we conclude that
\[
\left|\left((-1)^j\psi(t)D^{2s-2j}\binom{2j+1}{\ell}  \partial_x^{\ell}\varphi(x) \partial_x^{2j+1-\ell} u_n, u_n \right)_{L^2(0,T;L^2)}\right| \lesssim 1,\]
for $\ell = 2, \cdots, 2j+1$ and $n \in \N$. Consequently, we obtain
\[\left|\left((-1)^j\psi(t)D^{2s-2j} (\partial_x\varphi) \partial_x^{2j} u_n, u_n \right)_{L^2(0,T;L^2)}\right| \lesssim 1.\]
Taking the limit on $n$, we conclude \eqref{eq:proof3}.

\subsubsection*{\textbf{Proof of \eqref{eq:proof4}}} A straightforward computation gives
\[\begin{aligned}
\left(\psi(t)D^{2s-2j}\chi^2 \partial_x^{2j} u_n, u_n\right)_{L^2(0,T;L^2)} =&~{}\left(\psi(t)[D^{s-2j},\chi]\chi \partial_x^{2j} u_n, D^s u_n\right)_{L^2(0,T;L^2)}\\
&+\left(\psi(t)D^{s-2j}\chi \partial_x^{2j} u_n, [D^s, \chi] u_n\right)_{L^2(0,T;L^2)}\\
&+\left(\psi(t)D^{s-2j}\chi \partial_x^{2j} u_n, D^s\chi u_n\right)_{L^2(0,T;L^2)}\\
=:&~{}I+II+III.
\end{aligned}\]
Lemmas  \ref{lem:La1} and \ref{lem:La2} ensure for $u \in X_T^{r,b} \cap L_{loc}^2(0,T;H^s(\omega))$ that
\[\norm{\chi u_n}_{H^s} \le \norm{e^{\frac1n \partial_x^2}\chi u}_{H^s} + \norm{[\chi, e^{\frac1n \partial_x^2}]u}_{H^s} \lesssim \norm{\chi u}_{H^s} + \norm{u}_{H^{s-1}}\]
and
\[\norm{\chi \partial_x^{2j} u_n}_{H^{\sigma - 2j}} \le \norm{\chi u_n}_{H^{\sigma}} + \norm{[\chi, \partial_x^{2j}]u_n}_{H^{\sigma -2j}} \lesssim \norm{\chi u_n}_{H^{\sigma}} + \norm{u_n}_{H^{\sigma-1}},\]
for $\sigma \in \R$. The Cauchy-Schwarz inequality and Lemmas \ref{lem:La1} and \ref{lem:La2}, in addition to above estimates, yield
\[\begin{aligned}
|I| \le&~{} \norm{u_n}_{L^2(0,T;H^r)}\norm{\psi(t)D^{\rho}[D^{s-2j},\chi]\chi \partial_x^{2j} u_n}_{L^2(0,T;L^2)}\\
\le&~{} \norm{u_n}_{L^2(0,T;H^r)}\norm{\psi(t)\chi \partial_x^{2j} u_n}_{L^2(0,T;H^{s+\rho - 1- 2j})}\\
\lesssim&~{} \norm{u_n}_{X_T^{r,b}}\left(\norm{\psi(t) \chi u}_{L^2(0,T;H^s)} + \norm{u}_{X_T^{r,b}} + \norm{u_n}_{X_T^{r,b}}\right)\\
\lesssim&~{} 1,
\end{aligned}\]
\[\begin{aligned}
|II| \le&~{} \norm{\psi(t)\chi \partial_x^{2j} u_n}_{L^2(0,T;H^{s-2j-\rho})}\norm{D^{\rho}[D^s, \chi] u_n}_{L^2(0,T;L^2)} \\
\le&~{} \norm{\psi(t)\chi \partial_x^{2j} u_n}_{L^2(0,T;H^{s-2j-\rho})}\norm{u_n}_{L^2(0,T;H^{\rho +s -1})}\\
\lesssim&~{} \norm{u_n}_{X_T^{r,b}}\left(\norm{\psi(t) \chi u}_{L^2(0,T;H^s)} + \norm{u}_{X_T^{r,b}} + \norm{u_n}_{X_T^{r,b}}\right)\\
\lesssim&~{} 1
\end{aligned}\]
and
\[\begin{aligned}
|III| \le&~{} \int_0^T \psi(t) \norm{\chi \partial_x^{2j} u_n}_{H^{s-2j}}\norm{\chi u_n}_{H^s}\; dt\\
\lesssim&~{}\int_0^T \psi(t) \norm{\chi u_n}_{H^s}\left(\norm{\chi u_n}_{H^s} + \norm{u_n}_{H^{s-1}}\right)\; dt\\
\lesssim&~{} \norm{u}_{L^2(I;H^s(\omega))}^2 + \norm{u}_{X_T^{r,b}}^2 + \left(\norm{\psi(t)\chi u}_{L^2(0,T;H^s)} + \norm{u}_{X_T^{r,b}}\right)\norm{u_n}_{X_T^{r,b}}\\
\lesssim&~{}1.
\end{aligned}\]
Thus, we obtain
\[\left| \left(\psi(t)D^{2s-2j}\chi^2 \partial_x^{2j} u_n, u_n\right)_{L^2(0,T;L^2)} \right| \lesssim 1,\]
which implies \eqref{eq:proof4} by taking limit on $n$, and the proof of Proposition \ref{prop_b} is achieved.
\end{proof}

\subsection{Unique continuation property} As a consequence of the propagation of regularity, we prove the following unique continuation property for the higher-order KdV type. First, let us prove the auxiliary lemma.

\begin{lemma}
\label{unique_a}Let $u\in X^{0,\frac{1}{2}}_{T}$ be a solution of%
\begin{equation}
\partial_{t}u+(-1)^{j+1} \px^{2j+1}u + u\px u=0 \;\; \mbox{on} \;\; (0,T) \times \T
\text{.} \label{k34}%
\end{equation}
Assume that $u\in C^{\infty}\left(\left(  0,T\right) \times \omega  \right)
$, where $\omega\subset\mathbb{T}$ nonempty set. Then, $u\in C^{\infty}\left(\left(  0,T\right) \times \T  \right)$.
\end{lemma}

\begin{proof}
Recall that the mean value $\left[  u\right]  $ is conserved. Changing $a$
into $a+\left[  u\right]  $ if needed, we may assume that $\left[  u\right]
=0$. Using Lemma \ref{strichartz_estimates} (or from Lemma \ref{lem:Xsb}), we have that $u\partial_x u\in X^{0,-\frac{1}{2}}_{T}$. It follows from Proposition \ref{prop_b} with $f = - u \px u$ that%
\[
u\in L_{loc}^{2}( 0,T;H^{\frac{1}{2}}(\T))  \text{.}%
\]
Choose $t_{0}$ such that $u\left(  t_{0}\right)  \in H^{\frac{1}{2}}\left(
\mathbb{T}\right)  $. We can then solve (\ref{k34}) in $X^{\frac{1}{2}%
,\frac{1}{2}}_{T}$ with the initial data $u\left(  t_{0}\right)  $. By
uniqueness of solution in $X^{0,\frac{1}{2}}_{T}$, we conclude that $u\in
X^{\frac{1}{2},\frac{1}{2}}_{T}$. Applying Proposition
\ref{prop_b} iteratively,  we obtain%
\[
u\in L^{2}\left(  0,T;H^{r}\left(  \mathbb{T}\right)  \right)  \text{,
}\forall r \ge 0\text{,}%
\]
and, hence $u\in C^{\infty}\left(\left(  0,T\right) \times \T  \right)$.
\end{proof}

As a consequence of the previous result, we have the following unique continuation property.

\begin{corollary} \label{UCP}Let $\omega$ be a nonempty open set
in $\mathbb{T}$ and let $u\in X^{0,\frac{1}{2}}_{T}$ be a solution of%
\[
\left\{
\begin{array}
[c]{lll}%
\partial_{t}u+(-1)^{j+1} \px^{2j+1}u + u \px u=0 &  & \text{on } \;\; (0,T) \times \T,\\
u=c &  & \text{on } \;\; (0,T) \times \omega,%
\end{array}
\right.
\]
where $c\in\mathbb{R}$ denotes some constant. Then, $u(t,x)  =c$ on $(0,T) \times \T$. Furthermore, if the mean $[u]=0,$ then $
u(t,x)  =0$ on $(0,T) \times \T$.
\end{corollary}

\begin{proof}
Using Lemma \ref{unique_a}, we infer that $u\in C^{\infty}((0,T) \times \T)$. It follows that $u\equiv c$ on
$(0,T) \times \T$ by the unique property proved by Saut and Scheurer in \cite{Saut}.
\end{proof}

\section{Global stability: Proof of Theorem \ref{exponential_global}}\label{sec6}
In this section, we can establish the global results for the higher-order nonlinear dispersive equation (while Theorem \ref{stability_zhang} has a local aspect). The main ingredients are the propagation of singularities and the unique continuation property shown in the previous section.  Accurately, we are concerned with the stability properties of the closed loop system
\begin{equation}
\left\{
\begin{array}
[c]{lll}%
\partial_{t}u+(-1)^{j+1} \px^{2j+1}u+u\partial_xu=-K_{\lambda}u &  & \text{in } \{t>0\} \times \T \text{,}\\
u\left(  0,x\right)  =u_{0}\left(  x\right)  &  & \text{on }\mathbb{T}\text{,}%
\end{array}
\right.  \label{k35}%
\end{equation}
where $\lambda\geq0$ is a given number, $u_{0}\in H_{0}^{s}\left(  \mathbb{T}\right)$, for any $s\geq0$ and $K_{\lambda}$ is defined by  $K_{\lambda}u (t,x)  \equiv GG^{\ast}L_{\lambda}^{-1}u(t,x)$ for the operator $G$ defined as in \eqref{k5a}. It is known (see, for e.g. \cite{Laurent}) that $G$ is a linear bounded operator from $L^2(0,T;H^s_0(\mathbb{T}))$ into itself. Moreover, G is a self-adjoint positive operator on $L^2_0(\mathbb{T})$.

\subsection{Proof of Theorem \ref{exponential_global} for $s=0$}\label{sec:5.1} Consider $\lambda=0$ in $K_{\lambda}$ and remember that $K_0=GG^*$, so Theorem \ref{exponential_global} in $L^2$-level is a direct consequence of the following \textit{observability inequality}:

\vspace{0.2cm}

\noindent \textit{Let $T>0$ and $R_{0}>0$ be given. There exists a constant
$\mu>1$ such that for any $u_{0}\in L_{0}^{2}\left(  \mathbb{T}\right)  $
satisfying%
\[
\left\Vert u_{0}\right\Vert _{L^2}\leq R_{0}\text{,}%
\]
the corresponding solution $u$ of (\ref{k35}), with $\lambda=0$, satisfies%
\begin{equation}
\left\Vert u_{0}\right\Vert _{L^2}^{2}\leq\mu\int_{0}^{T}\left\Vert
Gu (t)\right\Vert _{L^2}^{2} \; dt\text{.} \label{k46}%
\end{equation}}

Indeed, suppose that \eqref{k46} holds. Assuming that $\lambda=0$, the energy estimate give us
\begin{equation}\label{eq:energy00}
\left\Vert u\left( T, \cdot\right)  \right\Vert _{L^2}^{2}=\left\Vert
u_{0}\right\Vert _{L^2}^{2}-\int_{0}^{T}\left\Vert Gu(t)\right\Vert _{L^2}^{2} \;  dt\text{,}
\end{equation}
which jointly with \eqref{k46} insures,
\[
\left\Vert u\left(  T, \cdot\right)  \right\Vert _{L^2}^{2}\leq\left(
\mu-1\right) \int_{0}^{T}\left\Vert Gu(t)\right\Vert _{L^2}^{2} \;  dt
\]
or equivalently, 
\[
\left\Vert u\left(  T, \cdot\right)  \right\Vert _{L^2}^{2}\leq\left(
\mu-1\right) \left(\left\Vert
u_{0}\right\Vert _{L^2}^{2}-\left\Vert u\left(  T, \cdot\right)  \right\Vert _{L^2}^{2}\right).
\]
Thus, 
\[
\left\Vert u\left(  T, \cdot\right)  \right\Vert _{L^2}^{2}\leq\frac{\mu-1}{\mu} \left\Vert
u_{0}\right\Vert _{L^2}^{2}.
\]
In this way, we inductively obtain that $$\|u( k T,\cdot)\|_{L^{2}}^{2} \leqslant \left(\frac{\mu-1}{\mu}\right)^{k}\left\|u_{0}\right\|_{L^{2}}^{2},$$ for all $k \geqslant 0 .$ Finally, analogously to \eqref{eq:energy00}, we know $$\|u( t, \cdot)\|_{L^{2}} \leqslant \|u( k T,\cdot)\|_{L^{2}},$$ for $k T \leqslant t \leqslant(k+1) T$, thus
\begin{equation}\label{exp_r}
\left\Vert u\left(  t, \cdot\right)  \right\Vert _{L^2(\mathbb{T})} \leqslant c e^{-\gamma t}\left\|u_{0}\right\|_{L^2(\mathbb{T})}, \quad \forall t \geqslant 0,
\end{equation}
where $c=\frac{\mu}{\mu-1}$ and $\gamma=\frac{\log \left(\frac{
\mu}{\mu-1}\right)}{T}$, and Theorem \ref{exponential_global}  holds true for $s=0$. \qed

\medskip

Let us now turn to prove inequality \eqref{k46}. To do that we
argue by contradiction. Suppose not, there exist a sequence
$\left\{  u_{n}\right\}  _{n\in\mathbb{N}}=u_{n}$, such that $u_{n}%
\in Y^{0}_{T}$ is solution of (\ref{k35}) satisfying%
\[
\left\Vert u_{0, n}  \right\Vert _{L^{2}}\leq R_{0}%
\]
but%
\begin{equation}
\int_{0}^{T}\left\Vert Gu_{n}(t)\right\Vert _{L^{2}}^{2} \; dt<\frac{1}{n}\left\Vert
u_{0,n}\right\Vert _{L^{2}}^{2}\text{,} \label{k47}%
\end{equation}
where $u_{0,n}=u_{n}\left(  0\right)  $. Let $\xi_{n}:=\left\Vert u_{0,n}\right\Vert
_{L^{2}}\leq R_{0}$. Then, one can choose a subsequence of $\xi_{n}=\left\{  \xi
_{n}\right\}  _{n\in\mathbb{N}}$, still denote by $\xi_{n}$, such that,
\[
\lim_{n\rightarrow\infty}\xi_{n}=\xi\text{.}%
\]
There are two possible cases: (a) $\xi>0$ and (b) $\xi=0$.

\medskip

\noindent \textbf{\textit{Case (a): $\xi>0$}}. Since the sequence $u_{n}$ is bounded in $Y^0_T$, by Lemma
\ref{lem:Xsb} (see particularly \cite[Lemma 2.4]{HongKwak}), the sequence $\left\{  \partial_{x}\left(  u_{n}%
^{2}\right)  \right\}  _{n\in\mathbb{N}}$ is bounded in $X^{0,-\frac{1}{2}
}_{T}$. 
By the compactness of embedding (taking subsequences if needed, but still denote by $u_{n}$), we know
\[u_{n}\rightarrow u  \quad \text{in }X^{-1,0}_{T} \quad \mbox{and} \quad -\frac{1}{2}\partial_{x}\left(  u_{n}^{2}\right)  \rightharpoonup f \quad
\text{ in }X^{0,-\frac{1}{2}}_{T},\]
where $u\in X^{0,\frac{1}{2}}_{T}$ and $f\in X^{0,-\frac{1}{2}}_{T}$.
Moreover, from \eqref{k11}, we obtain
\[
X^{0,\frac{1}{2}}_{T} \hookrightarrow L^{4} ((0,T) \times \T)  \text{,}%
\]
which ensures that $u_{n}^{2}$ is
bounded in $L^{2}\left((0,T) \times \T \right)  $.
Therefore, its follows that $\partial_{x}\left(  u_{n}^{2}\right)  $ is
bounded in
\[
X^{-1,0}_{T}=L^{2}\left(  0,T;H^{-1}\left(  0,L\right)  \right)  \text{.}%
\]
From interpolation of the space $X^{0,-\frac{1}{2}}_{T}$ and $X^{-1,0}_{T}$,
we obtain that $\partial_{x}\left(  u_{n}^{2}\right)  $ is bounded in
$X^{-\theta,-\frac{1}{2}+\frac{\theta}{2}}_{T}$, for $\theta\in\left[
0,1\right]  $. Again using the compactness of embedding
we conclude that
\[
-\frac{1}{2}\partial_{x}\left(  u_{n}^{2}\right)  \rightarrow f\text{ in
}X^{-1,-\frac{1}{2}}_{T}\text{.}
\]
On the other hand, it follows from (\ref{k47}) that%
\begin{equation}\label{k48a}
\int_{0}^{T}\left\Vert Gu_{n}\right\Vert _{L^2}^{2}dt\rightarrow\int_{0}%
^{T}\left\Vert Gu\right\Vert _{L^2}^{2}dt=0\text{,} %
\end{equation}
which implies from the definition of the operator $G$ as in \eqref{k5a} that 
\[u(t,x) = \int_\T g(y)u(t,y) \; dy =: c(t) \quad \mbox{on} \;\; (0,T) \times \omega,\]
where $\omega = \{x \in \T : g > 0\}$.
Thus, taking $n\rightarrow\infty$, we obtain from (\ref{k35}) that%
\[\left\{
\begin{array}
[c]{lll}%
\partial_{t}u+(-1)^{j+1} \px^{2j+1}u=f &  &
\text{on }(0,T) \times \T  \text{,}\\
u(t,x)  =c\left(  t\right)  &  & \text{on } (0,T) \times \omega  \text{.}%
\end{array}
\right.  \]
We prove now that $f=-\frac{1}{2}\partial_{x}\left(  u^{2}\right)  $. Indeed,
pick $w_{n}=u_{n}-u$ and $f_{n}=-\frac{1}{2}\partial_{x}\left(  u_{n}%
^{2}\right)  -f-K_{0}u_{n}$. 
Remark that from \eqref{k48a},
\begin{equation}
\int_{0}^{T}\left\Vert Gw_{n}\right\Vert _{L^2}^{2}dt=\int_{0}^{T}\left\Vert
Gu_{n}\right\Vert _{L^2}^{2}dt+\int_{0}^{T}\left\Vert Gu\right\Vert _{L^2}%
^{2}dt-2\int_{0}^{T}\left(  Gu_{n},Gu\right) dt\rightarrow0\text{.}
\label{k49}%
\end{equation}
Since $w_{n}\rightharpoonup0$ in $X^{0,\frac{1}{2}}_{T}$ we infer from Rellich theorem that%
\[
\int_{\mathbb{T}}g\left(  y\right)  w_{n}\left( t, y \right)  dy\rightarrow
0\text{ in }L^{2}\left(  0,T\right)  \text{.}%
\]
Combined with (\ref{k49}), this yields%
\[
\|gw_n\|_{L^2(0,T;L^2)} \le \|Gw_n\|_{L^2(0,T; L^2)} + \left\|\int_\T g(y) w_n(\cdot ,y) \; dy\right\|_{L^2(0,T)} \rightarrow 0.
\]
Then, $w_n$ and $f_n$ satisfy
\[
\partial_{t}w_{n} + (-1)^{j+1}\px^{2j+1} w_n=f_{n}%
\]%
and
\[f_{n}\rightarrow0\text{ in }X^{-1,-\frac{1}{2}}_{T}  \quad \mbox{and} \quad w_{n}\rightarrow0\text{ in }L^{2}\left(  0,T;L^{2}\left(  \tilde{\omega
}\right)  \right)  \text{,}%
\]
where $\tilde{\omega}:=\left\{  g>\frac{\left\Vert g\right\Vert _{L^{\infty
}  }}{2}\right\}  $.
Applying the Proposition \ref{prop_a} with $b=\frac{1}{2}$ and $b^{\prime}=0$, we conclude that
\[
w_{n}\rightarrow0\text{ in }L_{loc}^{2}\left(  0,T ; L^{2}\left(  \mathbb{T}\right)  \right)  \text{.}%
\]
Consequently, $u_{n}^{2}$ tends to $u^{2}$ in $L_{loc}^{1}\left(0,T ; L^{2}\left(  \mathbb{T}\right)  \right)  $ and $\partial
_{x}\left(  u_{n}^{2}\right)  $ tends to $\partial_{x}\left(  u^{2}\right)  $
in distributional sense. Therefore, $f=-\frac{1}{2}\partial_{x}\left(
u^{2}\right)  $ and $u\in X_T^{0, \frac{1}{2}}$ satisfies%
\[
\left\{
\begin{array}
[c]{lll}%
\partial_{t}u+(-1)^{j+1} \px^{2j+1}u+\frac
{1}{2}\partial_{x}\left(  u^{2}\right)  =0 &  & \text{on }(0,T) \times \T \text{,}\\
u\left(  t, x \right)  =c\left(  t\right)  &  & \text{on }(0,T) \times \omega  \text{.}%
\end{array}
\right.
\]
The first equation give $c^{\prime}\left(  t\right)  =0$ which, combined with
the unique continuation property (Corollary \ref{UCP}) ensures that $u(t,x) =c$, for some $c\in\mathbb{R}$. Since $\left[  u\right]  =0$, then
$c=0$ and%
\[
u_{n}\rightarrow0\text{ in }L_{loc}^{2}\left(  \left(  0,T\right)
,L^{2}\left(  \mathbb{T}\right)  \right)  \text{.}%
\]
To end the proof of {\bf Case (a)}, we take a particular time $t_{0}\in\left[  0,T\right]  $ such that%
\[
u_{n}\left(  t_{0}\right)  \rightarrow0\text{ in }L^{2}\left(
\mathbb{T}\right)  \text{.}%
\]
A direct computation gives
\[
\left\Vert u_{0,n}  \right\Vert _{L^2}^{2}=\left\Vert u_{n}\left(
t_{0}\right)  \right\Vert _{L^2}^{2}+\int_{0}^{t_{0}}\left\Vert Gu_{n}(t)%
\right\Vert _{L^2}^{2}dt\text{,}%
\]
and this makes a contradiction, since the right-hand side converges to $0$ while the left-hand side does not by the hypothesis.

\medskip

\noindent \textbf{\textit{Case (b): $\xi=0$}}. Note from \eqref{k47} that $\xi_{n}>0$, for all $n \in \N$. For each $n \in \N$, set $v_{n}=\frac{u_{n}}%
{\xi_{n}}$. Then $v_n$ satisfies
\[
\partial_{t}v_n+(-1)^{j+1} \px^{2j+1}v_n+K_{0}v_{n}+\frac{\xi_{n}}{2}\partial_{x}\left(  v_{n}%
^{2}\right)  =0%
\]%
with
\begin{equation}
\int_{0}^{T}\left\Vert Gv_{n}\right\Vert _{L^2}^{2}dt<\frac{1}{n}
\label{k50}
\end{equation}
and%
\[\left\Vert v_{0,n}  \right\Vert _{L^2}=1\text{.} \]
Analogously as above, by the compactness of embedding, $v_n$ (by extracting the subsequences if needed, but still denote by $v_n$) satisfies
\[
v_{n}\rightarrow v\text{ in }X^{-1,-\frac{1}{2}}_{T}\cap X^{-1,0}_{T} \quad \mbox{and} \quad \xi_{n}\partial_{x}\left(  v_{n}^{2}\right)  \rightarrow0\text{ in }%
X^{0,-\frac{1}{2}}_{T}\footnote{It follows from the boundedness of $\px(v_n^2)$ in $X_T^{0,\frac12}$ and $\xi_n \to 0$.}.%
\]
From (\ref{k50}), we have
\[
\int_{0}^{T}\left\Vert Gv\right\Vert _{L^2}^{2}dt=0\text{,}%
\]
which ensures that $v$ solves%
\begin{equation}
\left\{
\begin{array}
[c]{lll}%
\partial_{t}v+(-1)^{j+1} \px^{2j+1}v=0 &  &
\text{on } (0,T) \times \T  \text{,}\\
v\left(  t, x \right)  =c\left(  t\right)  &  & \text{on }(0,T) \times \omega  \text{.}%
\end{array}
\right.  \label{k52}%
\end{equation}
Thanks to Holmgren's uniqueness theorem (see e.g. \cite{hormander}), we conclude $c\left(  t\right)  =c\in\mathbb{R}$.
Moreover, as $\left[  v\right]  =0$ then
$c=0$. According to (\ref{k50}), $Gv_n$ converges to $0$ in $L^2(0,T;L^2)$, thus so
\[
K_{0}v_{n}\rightarrow0\text{ in }X^{-1,-\frac{1}{2}}_{T}\text{.}%
\]
Applying Proposition \ref{prop_a} as in {\bf Case (a)}, we have
\[
v_{n}\rightarrow0\text{ in }L_{loc}^{2}\left(0,T ; L^{2}\left(  \mathbb{T}\right)  \right)  \text{,}%
\]
thus we achieve the same conclusion, showing the result. \qed

\begin{remark}
In view of the hypotheses in Proposition \ref{prop_a}, the proof above is still valid, even if $f_n$ converges to $0$ only in larger class (e.g., $X^{-j, -\frac12}$). In other words, the property of propagation of compactness (Proposition \ref{prop_a}) established in this work enables one to extend Theorem \ref{exponential_global} for rougher solutions.

\end{remark}

\subsection{Proof of Theorem \ref{exponential_global}} Once again, consider $\lambda = 0$ in  $K_{\lambda}$. Now we prove that the solution $u$ of \eqref{k35} decays exponentially in $H^s$-level for any $s>0$. We first prove it when $s = 2j+1$. Then, interpolating with the result in Section \ref{sec:5.1}, we obtain the conclusion for $0 \le s \le 2j+1$. The similar argument can be applied for $s = (2j+1)\N$, and thus we complete the proof.

\medskip

Fix $s=2j+1$, $j\in\mathbb{N}$, and pick any $R>0$ and any $u_0\in H^{2j+1}(\mathbb{T})$ with $\left\Vert u_{0}\right\Vert _{L^2\left(  \mathbb{T}\right)  }\leq R$. Let $u$ solution of \eqref{k35} with initial condition $u_0$, and set $v=u_t$. Then, $v$ satisfies
\begin{equation}
\left\{
\begin{array}
[c]{lll}%
\partial_{t}v+(-1)^{j+1} \px^{2j+1}v+\partial_x(uv)=-K_0v &  & \text{in } \{t>0\} \times \T \text{,}\\
v\left(  0,x\right)  =v_{0}\left(  x\right)  &  & \text{on }\mathbb{T}\text{,}%
\end{array}
\right.  \label{k35_r}%
\end{equation}
where 
\begin{equation}\label{eq:v_0}
v_0 = -K_0u_0 -u_0u'_0 -(-1)^{j+1}\partial_x^{2j+1}u_0.
\end{equation}
According to Theorem \ref{thm:GWP} and the exponential decay \eqref{exp_r}, for any $T > 0$ there exists constants  $C > 0$ and $\gamma > 0$ depending only on $R$ and $T$ such that
$$
\|u(\cdot, t)\|_{Y^{0}_{[t, t+T]}} \leq C e^{-\gamma t}\left\|u_{0}\right\|_{L^2} \quad \mbox{for all} \;\; t \geq 0.
$$
Thus, for any $\epsilon>0$, there exists a $t^{*}>0$ such that if $t \geq t^{*}$, one has
$$
\|u(\cdot, t)\|_{Y^{s}_{[t, t+T]}} \leq \epsilon.
$$
At this point we need an exponential stability result for the linearized system
\begin{equation}
\left\{
\begin{array}
[c]{lll}%
\partial_{t}w+(-1)^{j+1} \px^{2j+1}w+\partial_x(aw)=-K_0w &  & \text{in } \{t>0\} \times \T \text{,}\\
w\left(  0,x\right)  =w_{0}\left(  x\right)  &  & \text{on }\mathbb{T}\text{,}%
\end{array}
\right.  \label{k35_ra}%
\end{equation}
where $a \in Y^s_T$ is a given function.
\begin{lemma}\label{a_function}
Let $s \geq 0$ and $a \in Y^s_T$ for all $T>0 .$ Then for any $\gamma^{\prime} \in(0, \gamma)$ there exist $T>0$ and $\beta>0$ such that if
\begin{equation}\label{eq:hypothesis a}
\sup _{n \geq 1}\|a\|_{Y^{s}_{[n T, (n+1) T]}} \leq \beta
\end{equation}
then
$$
\|w(\cdot, t)\|_{H^s} \lesssim e^{-\gamma^{\prime} t}\left\|w_{0}\right\|_{H^s} \quad \mbox{for all} \;\;  t \geq 0.
$$
Remark that the implicit constant depends on $\|w_0\|_{H^s}$, but not $w_{0}$. 
\end{lemma}
Suppose Lemma \ref{a_function} is valid. Choose $\epsilon<\beta$, and then apply Lemma \ref{a_function} to \eqref{k35_r} to obtain
$$
\|v(\cdot, t)\|_{L^2} \lesssim e^{-\gamma^{\prime}\left(t-t^{*}\right)}\left\|v\left(\cdot, t^{*}\right)\right\|_{L^2},
$$
for any $t \geq t^{*}$, or equivalently
\[\|v(\cdot, t)\|_{L^2} \lesssim e^{-\gamma^{\prime} t}\left\|v_{0}\right\|_{L^2},\]
for any $t \geq 0$.
From
$$
(-1)^{j+1} \px^{2j+1}u=-K_{0} u-u\partial_{x} u-v,
$$ 
a direct computation gives 
\[\|u(t)\|_{H^s} \lesssim \|u(t)\|_{L^2} + \|u(t)\|_{L^2}\|u(t)\|_{H^s} + \|v(t)\|_{L^2}.\]
Applying Theorem \ref{exponential_global} for $s=0$ established in Section \ref{sec:5.1}, Theorem \ref{thm:GWP} and Lemma \ref{a_function} with \eqref{eq:v_0} to the right-hand side, we obtain
$$
\|u(\cdot, t)\|_{H^s} \lesssim C e^{-\gamma^{\prime} t}\left\|u_{0}\right\|_{H^s},
$$
for any $t \geq 0$. Note that the implicit constant here depends only on $R$. This proves Theorem \ref{exponential_global} for $s = 2j+1$. Moreover applying Lemma \ref{a_function}  for $ w = u_{1}-u_{2}$ and $a=u_{1}+u_{2}$ when $u_1, u_2$ are two different solutions, we obtain the Lipchitz stability estimate, which is required for interpolation:
$$
\left\|\left(u_{1}-u_{2}\right)(\cdot, t)\right\|_{0} \leq C e^{-\gamma^{\prime} t}\left\|\left(u_{1}-u_{2}\right)(\cdot, 0)\right\|_{0}.
$$
Thus, it remains to prove Lemma \ref{a_function}. \qed
\begin{proof}[Proof of Lemma \ref{a_function}] 
Let $T > 0$ and $s \ge 0$ be given, and $a \in Y^s_T$. Similarly as the proof of Theorem \ref{thm:GWP}, we can show that the system \eqref{k35_ra} admits a unique solution $Y_T^s \cap C_TH^{s}$, and the solution $u$ satisfies
\begin{equation}\label{eq:global bound r}
\norm{u}_{Y_T^s} \le \alpha_{T,s}(\|a\|_{Y^s_T})\norm{u_0}_{H^s},
\end{equation}
where $\alpha_{T,s}$ is positive nondecreasing continuous function. By Duhamel's principle, the solution $u$ to \eqref{k35_ra} is equivalent to the following integral form:
$$
w(t)=W_{0}(t) w_{0}-\int_{0}^{t} W_{0}(t-s) \partial_{x}(a w)(s) \; ds,
$$
where $W_{0}(t)=e^{-t\left((-1)^{j+1} \px^{2j+1}+K_{0}\right)} .$ Then, thanks to Proposition \ref{zhang_2}, Lemma \ref{lem:modified Xsb} and \eqref{eq:global bound r}, we get
\begin{equation}\label{eq:GWP estimate}
\|w(\cdot, T)\|_{H^s} 
\leq C_{1} e^{-\gamma T}\left\|w_{0}\right\|_{H^s}+C_{2}\|a\|_{Y^{s}_{T}}\alpha_{T,s}\left(\|a\|_{Y^s_T}\right)\left\|w_{0}\right\|_{H^s},
\end{equation}
where $C_{1}>0$ is independent of $T$ while $C_{2}$ may depend on $T$. Let
$$
y_{n}=w(\cdot, n T) \quad \mbox{for} \;\; n \in \N.
$$
Then, similarly as \eqref{eq:GWP estimate}, we obtain for each $n \in \N$ that
$$
\left\|y_{n+1}\right\|_{H^s} \leq C_{1} e^{-\gamma T}\left\|y_{n}\right\|_{H^s}+C_{2,n}\|a\|_{Y^{s}_{[n T,(n+1) T]}} \alpha_{T,s}\left(\|a\|_{Y^{s}_{[n T,(n+1) T]}}\right)\left\|y_{n}\right\|_{H^s}.
$$
By choosing appropriate $T>0$ large enough and $\beta>0$ small enough such that
$$
C_{1} e^{-\gamma T}+C_{2} \beta \alpha_{T,s}(\beta)=e^{-\gamma^{\prime} T},
$$
we conclude that
\begin{equation}\label{eq:n decay}
\left\|y_{n+1}\right\|_{H^s} \leq e^{-\gamma^{\prime} T}\left\|y_{n}\right\|_{H^s}
\end{equation}
for any $n \geq 1$ as long as \eqref{eq:hypothesis a} is assumed. 
By using \eqref{eq:n decay} inductively, we obtain
$$
\left\|y_{n}\right\|_{H^s} \leq e^{-n \gamma^{\prime} T}\left\|y_{0}\right\|_{H^s}
$$
for any $n \geq 1$, which implies that
$$
\|w(\cdot, t)\|_{H^s} \leq C e^{-\gamma^{\prime} t}\left\|w_{0}\right\|_{H^s}
$$
for all $t \geq 0$. This completes the proof.
\end{proof}
\section{Concluding remarks and open issues} \label{sec7}
In this work we treat the global control issues for the general higher-order KdV type equation on periodic domain
\begin{equation}
\begin{cases}
\partial_{t}u+(-1)^{j+1} \px^{2j+1}u+u\partial_xu=Gh,\\
u(0,x) = u_0,
\end{cases}
\quad (t,x) \in (0,T) \times \T,
\label{k35a}%
\end{equation}
where, $Gh$ is defined as \eqref{k5a} and can take also the form $Gh:=K_{\lambda}u (t,x)  \equiv GG^{\ast}L_{\lambda}^{-1}u(t,x)$. The results presented in the manuscript recovered previous global control problems for the KdV and Kawahara equations, when $j=1$ and $2$, respectively. Nevertheless, presents global control results for a general KdV type equation, which is more complex than the studies previously presented. 

Precisely, thanks to the smoothing properties of solutions in Bourgain spaces we are able to prove the \textit{Strichartz estimates} and \textit{propagation of singularities} associated with the solution of the linear system of \eqref{k35a}. With this in hand, we prove an \textit{observability inequality} for the solutions of the system\eqref{k35a}.  This helps us prove the main results of the article. Even though it has a generalist character, the work presents interesting problems from the mathematical point of view, which we will detail below.

\subsection{Time-varying feedback law} A natural question that arises is related to global stabilization with an arbitrary large decay rate. This can be obtained by using a time-varying feedback law. As for the KdV and Kawahara equations, the time-varying feedback control law  for the higher-order KdV type equation can be found. Precisely,  it is possible to construct a continuous time-varying feedback law $K \equiv K(u,t)$ such that a semi-global stabilization holds with an arbitrary large decay rate in the Sobolev space $H^{s}(\mathbb{T})$ for any $s\geq 0.$ In fact, $K$ has the following form
$$K(u,t):= \rho\left(\|u\|^{2}_{H^{s}(\mathbb{T})}\right)\left[\theta\left(\frac{t}{T}\right) K_{\lambda}(u)+\theta\left(\frac{t}{T}-T\right) G G^{\ast}u\right] + \left(1-\rho\left(\|u\|^{2}_{H^{s}(\mathbb{T})}\right)\right)G G^{\ast}u,  $$
where $\rho \in C^{\infty}(\mathbb{R}^{+};[0,1])$ is a function such that
for some $r_{0}\in (0,1),$ we have
\[\rho(r):=\left\{
\begin{array}
[c]{lll}%
1, &  &
\text{for } r\leq r_{0} ,\\
0,  &  & \text{for } r\geq 1
\end{array}
\right.  \]
and $\theta \in C^{\infty}(\mathbb{R};[0,1])$ is a function with the following properties: $\theta(t+2)=\theta(t)$, for all $t\in \mathbb{R}$ and 
\[\theta(t):=\left\{
\begin{array}
[c]{lll}%
1,  &  & \text{for } \delta \leq t \leq 1-\delta \text{,}\\
0,  &  & \text{for } 1 \leq t \leq 2 \text{,}%
\end{array}
\right.\]
for some $\delta \in (0,\frac{1}{10}).$ Then, the following result holds true.
\begin{theorem} Let $\lambda>0$ and let $K=K(u, t)$ be as above. Pick any $\lambda^{\prime} \in(0, \lambda)$ and any $\lambda^{\prime \prime} \in\left(\lambda^{\prime} / 2,\left(k+\lambda^{\prime}\right) / 2\right) .$ Then there exists a time $T_{0}>0$ such that for $T>T_{0}$, $t_{0} \in \mathbb{R}$ and $u_{0} \in H^{s}(\mathbb{T})$, the unique solution of the closed-loop system
$$
\partial_{t}u+(-1)^{j+1} \px^{2j+1}u+u\partial_xu=-K(u, t),\quad (t,x) \in \R \times \T
$$
satisfies
$$
\|u(\cdot, t)-[u_{0}]\|_{H^s} \leq \gamma_{s, \mu}\left(\left\|u_{0}-[u_{0}]\right\|_{H^s}\right) e^{-\lambda^{\prime \prime}\left(t-t_{0}\right)}\left\|u_{0}-[u_{0}]\right\|_{H^s}, \text { for all } t>t_{0},
$$
where $\gamma_{s}$ is a nondecreasing continuous function.
\end{theorem}

\subsection{Low regularity control results} Observe that the results presented in this work are verified in $H^s(\mathbb{T})$, when $s\geq0$. However, by comparing with \cite{HongKwak}, a natural question appears.

\vspace{0.2cm}

\noindent\textbf{Problem $\mathcal{A}$:} Is it possible to prove control results for the system \eqref{k35a} with $-j/2\le s<0$? 

\vspace{0.2cm}

The answer for this question may be very technical and the well-posedness result probably will be the biggest challenge. Additionally, the unique continuation property needs to be proved and appears to be also a hard problem. Thus, the following open issues naturally appear. 

\vspace{0.2cm}

\noindent\textbf{Problem $\mathcal{B}$:} Is the system \eqref{k35a} globally well-posed in $H^s(\mathbb{T})$, for $-j/2 \le s<0$? 

\vspace{0.2cm}

\noindent\textbf{Problem $\mathcal{C}$:} Is the unique continuation property, presented in Lemma \ref{unique_a},  true for $-j/2 \le s<0$?

\appendix
\renewcommand*{\thetheorem}{\Alph{section}.\arabic{theorem}}
\setcounter{theorem}{0}

\section{Controllability and stability results: Linear problems} \label{sec3}
Let us consider the linear  open loop control system
\begin{equation}\label{k8}
\left\{\begin{array}{ll}
\pt u + (-1)^{j+1} \px^{2j+1}u =Gh, \hspace{1em} &(t,x) \in \R \times \T, \\
u(0,x) = u_0(x), &x \in \T,
\end{array}
\right.
\end{equation}
where the operator $G$ is defined as in \eqref{k5a} and $h=h(t,x)$ is the control input.

Consider the $L^2$--basis $\{\phi_k\}_{k\in\mathbb{Z}}$, thus the solution $u$ of  \eqref{k8} can be expressed in the form%
\begin{equation}
u(t,x)  =\sum_{k \in \Z}\left(  e^{\lambda_{k}%
t}u_{0,k}+\int_{0}^{t}e^{\lambda_{k}\left(  t-\tau\right)  }G_{k}\left[
h\right]  \left(  \tau\right)  d\tau\right)  \phi_{k}\left(  x\right),
\label{k9}%
\end{equation}
where $u_{0,k}$ are the Fourier coefficients of
$u_{0}$ and $G[h]$ are
\begin{equation}\label{eq:coefficients}
u_{0,k}=(u_{0},\phi_{k}) \quad \mbox{and} \quad G_{k}[h]=(Gh,\phi
_{k})=(h,G\phi_{k})%
\end{equation}
for $k \in \Z$, respectively. Moreover, for given $s\in\mathbb{R}$, if
$u_{0}$ $\in H^{s}\left(  \mathbb{T}\right)  $ and $h\in L^{2}\left(
0,T;H^{s}\left(  \mathbb{T}\right)  \right)  $, the function given by
(\ref{k9}) belongs to the space $C([0,T];H^{s}(\mathbb{T}))$. Now, we are in position to prove control results for the system \eqref{k8}.

\subsection{Controllability result} The first result means that system (\ref{k8}) is exactly controllable in time $T>0$ and can be read as follows.

\begin{theorem}\label{thm:EC}
Let $T>0$ and $s\in\mathbb{R}$ be given. There exists a bounded
linear operator%
\[
\Phi:H^{s}(\mathbb{T})\times H^{s}(\mathbb{T})\rightarrow L^{2}\left(
0,T;H^{s}(\mathbb{T})\right)
\]
such that for any $u_{0},u_{1}\in H^{s}(\mathbb{T})$ with $[u_0] = [u_1]$, if one chooses
$h=\Phi\left(  u_{0},u_{1}\right)  $ in (\ref{k8}), then the system (\ref{k8})
admits a solution $u\in C([0,T];H^{s}(\mathbb{T}))$ satisfying%
\[
\left.  u\right\vert _{t=0}=u_{0}\text{, \ }\left. u\right\vert _{t=T}%
=u_{1}\text{.}%
\]
Moreover, we have
\begin{equation}\label{eq:phi}
\norm{\Phi(u_0,u_1)}_{L^2(0,T;H^s)} \lesssim (\norm{u_0}_{H^s} + \norm{u_1}_{H^s}),
\end{equation}
here the implicit constant depends only on $T$, $\norm{g}_{H^s}$ and $\norm{g}_{H^{-s}}$.
\end{theorem}
\begin{remark}\label{rem:proof}
The proof of Theorem \ref{thm:EC} is standard provided that the eigenvalues of the associated linear operator satisfy Lemma \ref{lem:gap}, precisely, the proof relies only on the fact that the dual basis of $\{e^{\lambda_k t}\}$ is a Riesz  sequence, which follows from a classical theorem of Ingham and Beurling \cite{Ing1936,Beu}.
\end{remark}

\begin{proof}[Proof of Theorem \ref{thm:EC}]
As mentioned in Remark \ref{rem:proof}, the proof of Theorem \ref{thm:EC} is now standard, and can be found in the literature, for instance \cite{Laurent,Russel96,zhang1}. However, we also give a proof for the sake of self-containedness.

Since the solution $u(t,x)$ can be expressed as in \eqref{k9}, it suffices to find $h \in L^2(0,T;H^s(\mathbb{T}))$ such that
\[u_1(x) = \sum_{k \in \Z}\left(  e^{\lambda_{k}%
T}u_{0,k}+\int_{0}^{T}e^{\lambda_{k}\left(  T-\tau\right)  }G_{k}\left[
h\right]  \left(  \tau\right)  d\tau\right)  \phi_{k}\left(  x\right),\]
which follows from
\begin{equation}\label{eq:EC-1}
e^{-\lambda_kT} u_{1,k} - u_{0,k} = \int_0^T e^{-\lambda_k \tau} G_{k}[ h ] (\tau) \; d\tau, \quad \mbox{for each} \;\; k \in \Z,
\end{equation}
where $u_{0,k}$, $u_{1,k}$ and $G_{k}[ h ]$ are the Fourier coefficients defined as in \eqref{eq:coefficients}.

Note that  $\mathcal P := \{p_k = e^{\lambda_k t} : k \in \Z\}$ forms a Riesz basis for its closed span $\mathcal P_T$ in $L^2(0,T)$, and there uniquely exists the dual basis $\mathcal Q = \{q_k : k \in \Z\}$ in $\mathcal P_T$ such that
\begin{equation}\label{eq:orthonormal}
\int_0^T q_j(t) \overline{p_k(t)} \; dt = \delta_{jk}, \quad -\infty < j, k < \infty.
\end{equation}
We take the control input $h$ of the form
\begin{equation}\label{eq:h}
h(t,x) = \sum_{j \in \Z} h_j q_j(t)(G \phi_j)(x),
\end{equation}
where the coefficients $h_j$ are to be precisely determined later, depending on given $v_0$, $v_1$ and $g$.

Inserting \eqref{eq:h} into \eqref{eq:EC-1} and using \eqref{eq:orthonormal} and the fact that $G$ is a self-adjoin operator, one has
\begin{equation}\label{eq:h1}
e^{-\lambda_kT} u_{1,k} - u_{0,k} =\sum_{j \in \Z} \int_0^T e^{-\lambda_k t}h_j q_j(t)  (G(G \phi_j), \phi_k) dt = h_k \|G\phi_k\|_{L^2}^2.
\end{equation}
We set $\beta_k := \|G\phi_k\|_{L^2}^2$. The definitions of $\phi_k$ and $G\phi_k$ ensure that $\beta_k > 0$ for all $k \neq 0$. Moreover, a direct computation in \eqref{k5a} gives
\[\begin{aligned}
\beta_k =&~{}\frac{1}{2\pi}\int_{\T}g(x)^2 \; dx -2 \mbox{Re} \left(\left(\int_{\T} g(x) \phi_k(x) \; dx \right)\int_{\T}g(x)^2 \phi_{-k}(x) \; dx\right)\\
&~{}+\left|\int_{\T} g(x) \phi_k(x) \; dx \right|^2 \int_{\T} g(x)^2 \; dx,
\end{aligned}\]
which, in addition to Riemann-Lebesgue lemma, ensures
\[\lim_{|k| \to \infty} \beta_k = \frac{1}{2\pi}\int_{\T}g(x)^2 \; dx > 0.\]
From above observation, it follows that there exists $\delta >0$ such that
\[\beta_k > \delta >0, \quad k \neq 0.\]
Thus, from \eqref{eq:h1}, $h_k$ is naturally defined as
\[h_0 = 0 \quad \mbox{and} \quad h_k = \frac{e^{-\lambda_kT} u_{1,k} - u_{0,k}}{\beta_k}, \quad k \neq 0.\]

The rest of proof is to show that $h$ is in $L^{2}\left(0,T;H^{s}(\mathbb{T})\right)$ for all $u_0, u_1 \in H^s$. We write $G\phi_j$ and $g$ with the standard basis $\{\phi_n\}$ as
\[G\phi_j(x) = \sum_{n \in \Z} \mathcal G_{j, n}\phi_n(x) \quad \mbox{and} \quad g(x) = \sum_{n \in \Z} g_n \phi_n(x),\]
where $\mathcal G_{j, n} = (G\phi_j , \phi_n)$ and $g_n = (g, \phi_n)$, for all $j, n \in \Z$. Then, $h$ in \eqref{eq:h} can be rewritten as
\[h(t,x) = \sum_{j \in \Z} \sum_{n \in \Z} h_j q_j(t)\mathcal G_{j, n}\phi_n(x).\]
By using a classical theorem of Ingham and Beruling \cite{Ing1936,Beu}, a computation with the property of the Riesz basis $\mathcal Q$ gives
\[\begin{aligned}
\|h\|_{L^2(0,T:H^s)}^2 
\lesssim &~{} \sum_{n \in \Z} \bra{n}^{2s}\sum_{j \in \Z \setminus \{0\}} |h_j \mathcal G_{j, n}|^2 \\
=&~{} \sum_{j \in \Z \setminus \{0\}}|h_j|^2 \sum_{n \in \Z} \bra{n}^{2s}|\mathcal G_{j, n}|^2.
\end{aligned}\]
Using definitions of $\mathcal G_{j,n}$ and $G \phi_j$, one knows
\[\begin{aligned}
\mathcal G_{j, n} 
= (G\phi_j , \phi_n) 
= \sum_{k \in Z}g_k(\phi_k\phi_j, \phi_n) - g_{-j}g_n
= \frac{1}{\sqrt{2\pi}}g_{n-j} - g_{-j}g_n,
\end{aligned}\]
hence we obtain
\[|\mathcal G_{j, n}|^2 \lesssim |g_{n-j}|^2 + |g_{-j}|^2 |g_n|^2,\]
and conclude
\[ \sum_{n \in \Z} \bra{n}^{2s}|\mathcal G_{j, n}|^2 \lesssim  \sum_{n \in \Z} \bra{n+j}^{2s}|g_n|^2 +  |g_{-j}|^2\sum_{n \in \Z} \bra{n}^{2s}|g_n|^2.\]
When $s \ge 0$, a direct computation yields
\[\begin{aligned}
\|h\|_{L^2(0,T:H^s)}^2 \lesssim&~{}\sum_{j \in \Z \setminus \{0\}}|h_j|^2\left( \sum_{n \in \Z} \bra{j+n}^{2s}|g_n|^2 +  |g_{-j}|^2\sum_{n \in \Z} \bra{n}^{2s}|g_n|^2\right)\\
\lesssim&~{} \norm{g}_{H^s}^2\sum_{j \in \Z \setminus \{0\}}(\bra{j}^{2s} + |g_{-j}|^2)|h_j|^2\\
\lesssim&~{} \norm{g}_{H^s}^2\sum_{j \in \Z \setminus \{0\}}(\bra{j}^{2s} + |g_{-j}|^2)\beta_j^{-2}|e^{-\lambda_jT} u_{1,j} - u_{0,j}|^2\\
\lesssim&~{}\max_{j \in \Z \setminus \{0\}}\beta_{j}^{-2} \norm{g}_{H^s}^2(1+\norm{g}_{L^2}^2)\left(\norm{u_0}_{H^s}^2 + \norm{u_1}_{H^s}^2\right).
\end{aligned}\]
On the other hand, when $s < 0$, thanks to the crude estimate
\[\bra{a+b}^c \lesssim \bra{a}^{|c|}\bra{b}^c, \quad a,b,c \in \R,\]
we have similarly as before
\[\begin{aligned}
\sum_{n \in \Z} \bra{j}^{-2s}\bra{n}^{2s}|\mathcal G_{j, n}|^2 \lesssim&~{}  \sum_{n \in \Z} \bra{n}^{-2s}|g_n|^2 +  \bra{-j}^{-2s}|g_{-j}|^2\sum_{n \in \Z} \bra{n}^{2s}|g_n|^2 \\
\lesssim&~{}(1+\norm{g}_{H^s}^2)\norm{g}_{H^{-s}}^2,
\end{aligned}\]
and hence
\[\begin{aligned}
\|h\|_{L^2(0,T:H^s)}^2 \lesssim&~{} \sum_{j \in \Z \setminus \{0\}}\bra{j}^{2s}|h_j|^2 \sum_{n \in \Z}\bra{j}^{-2s} \bra{n}^{2s}|\mathcal G_{j, n}|^2\\
\lesssim&~{}\max_{j \in \Z \setminus \{0\}}\beta_{j}^{-2}(1+\norm{g}_{H^s}^2)\norm{g}_{H^{-s}}^2\left(\norm{u_0}_{H^s}^2 + \norm{u_1}_{H^s}^2\right).
\end{aligned}\]
Therefore, we complete the proof.
\end{proof}

As a consequence of Theorem \ref{thm:EC}, we have the following property for the unitary operator group $W$.
\begin{corollary}
Let $T > 0$ be given. Then, there exists $\delta > 0$ such that
\[\int_0^T \norm{GW(t)f}_{L^2}^2 \; dt \ge \delta \norm{f}_{L^2}^2,\]
for any $f \in L^2$.
\end{corollary}

\subsection{Feedback stabilization}
This part of the work gives a positive answer to the stabilization problem. Let us remember that if $s \in \R$, for any $\lambda>0$, we define a bounded linear operator from $H^s(\mathbb{T})$ to itself by
\[ L_{\lambda} f =\int_0^1e^{-2\lambda \tau} W (-\tau )  G G^{\ast} W^{\ast}(-\tau) f \;  d\tau,
\]
for any $f \in H^{s}\left(  \mathbb{T}\right) $. Note that $L_{0} \equiv I$. It is known that $L_{\lambda}$ is a self-adjoint positive operator on $H_0^s(\T)$ and so is its inverse $L_{\lambda}^{-1}$, for all $s \ge 0$ (see, for instance, \cite[Lemma 2.4]{Laurent}).  This fact enables us to take the control function $h(t,x) = -G^*L_{\lambda}^{-1} u(t,x)$, and by employing the following feedback control law
\[K_{\lambda}u (t,x)  \equiv GG^{\ast}L_{\lambda}^{-1}u(t,x)\]
we obtain the closed-loop system from \eqref{k8}, namely
\begin{equation}\label{eq:closed loop}
\pt u + (-1)^{j+1} \px^{2j+1}u =-K_{\lambda}u.
\end{equation}

We will give a stabilizability result, to see this we rewrite system \eqref{eq:closed loop} as an abstract control system in the Hilbert space $\mathcal{V}$:
\begin{equation}\label{eq:abstract}
\partial_t u=Au+Bh, \quad u(0)=u_{0},
\end{equation}
where A is the operator defined by \eqref{eq:operator} which corresponds to the continuous unitary operator group $W(t)$ on the space $L^{2}%
(\mathbb{T})$ satisfying \eqref{eq:AW}.
The following theorem is derived from Theorem \ref{thm:EC} and a classical principle \textit{exact controllability implies exponential stabilizability} for conservative control systems. For details, we suggest for the reader the references \cite{Liu,Slemrod}.

\begin{theorem}\label{th:stabilizability}
Assume that the assumptions of Theorem \ref{thm:EC} are satisfied. Then
\begin{itemize}
\item[(i)] There exist a $T>0$ and $\delta>0$ such that $$\int_0^T\norm{B^{\ast}W^{\ast}(t)u_0}^2_{L^2(\mathbb{T})}dt\geq\delta\norm{u_0}^2_{H^s(\mathbb{T})},$$
for any $u_0\in H^s$.
\item[(ii)] For any given $\lambda>0$, there exists an operator $K\in \mathcal{L}(H^s(\T),L^2(\T))$ such that if one chooses $h=Ku$ in \eqref{eq:abstract}, then the resulting closed-loop system
\begin{equation*}
\partial_t u=Au+BKu, \quad u(0)=u_{0},
\end{equation*}
has the property that its solution satisfies
\begin{equation*}
\norm{u(t)}_{H^s} \lesssim e^{-\lambda t}\norm{u_0}_{H^s},
\end{equation*}
\end{itemize}
\end{theorem}
In our context, by using the result due \cite{Liu,Slemrod}, the following proposition presents that the closed-loop system \eqref{eq:closed loop} is exponentially stable:
\begin{proposition}
\label{zhang_2}Let $s\geq0$ and $\lambda>0$ be given. Then for any $u_{0}\in H^{s}\left(  \mathbb{T}\right)  $, the linear closed-loop system \eqref{eq:closed loop} admits a unique solution $u\in C\left(  \left[  0,T\right]  ;H^{s}\left(  \mathbb{T}\right) \right)  $. Moreover, the solution $u$ obeys the following decay property:
\[\norm{u(t)-[u_{0}]}_{H^s} \lesssim e^{-\lambda t}\norm{u_0-[u_{0}]}_{H^s},\]
for any $t>0$. The implicit constant depends only on $s$.
\end{proposition}
\begin{proof}
The proof is the same as that of Theorem \ref{th:stabilizability}.
\end{proof}

\section{Well-posedness results}\label{sec4}

\subsection{Global well-posedness for the closed loop system}
Recall the system \eqref{eq:closed loop} with the nonlinearity $uu_x$
\begin{equation}\label{eq:nonlinear closed loop}
\pt u + (-1)^{j+1} \px^{2j+1}u +uu_x=-K_{\lambda}u.
\end{equation}
and its integral formulas
\begin{equation}\label{eq:duhamel}
\begin{aligned}
u(t)=&~{}W(t) u_{0}-\int_{0}^{t} W(t-\tau)(K_{\lambda} u)(\tau) d \tau-\int_{0}^{t} W(t-\tau)\left(u u_{x}\right)(\tau) d \tau\\
=:&~{}W_{\lambda}(t) u_{0} - \int_{0}^{t} W_{\lambda}(t-\tau)\left(u u_{x}\right)(\tau) d \tau,
\end{aligned}
\end{equation}
where $W_{\lambda}(T)$ is the linear propagator associated to \eqref{eq:closed loop}. Using Lemma \ref{lem:Xsb} (5) (with small modification) and the boundedness of $G, G^*$ and $L_{\lambda}^{-1}$, we immediately obtain
\begin{equation}\label{eq:K_lambda}
\left\| \int_a^t W(t-\tau)(K_{\lambda} u)(\tau)\;d\tau \right\|_{Y_I^s} \lesssim |I|^{1-\epsilon} \norm{u}_{Y_I^s},
\end{equation}
for $I = [a,b]$ with $0 < b-a < 1$ and $0 < \epsilon < \frac12$.

We now establish a similar result to the Lemma \ref{lem:Xsb} (2) and (3), associated to the propagator $W_{\lambda}$ but in $Y^s$.
\begin{lemma}\label{lem:modified Xsb}
Let $T>0$ be given.
\begin{enumerate}
\item For all $s \in \R$, we have for $f \in H^s$
\[\norm{W_{\lambda}(t)f}_{Y_T^s} \lesssim \norm{f}_{H^s}.\]

\item For all $s \in \R$, we have for $F \in Y_T^s$
\[\left\|\int_0^t W_{\lambda}(t-\tau)F(\tau)\; d\tau \right\|_{Y_T^s} \lesssim \norm{\mathcal F^{-1}\left(\bra{\tau - k^{2j+1}}^{-1}\widetilde F\right)}_{Y_T^s}.\]
\end{enumerate}
The implicit constants depends on $T$ and $s$.
\end{lemma}

\begin{proof}
For given $f \in H^s$ and $F \in Y_T^s$, set
\[u(t,x) = W_{\lambda}(t) f + \int_0^t W_{\lambda}(t-\tau) F(\tau)\; d\tau.\]
Then, $u$ solves
\[\pt u + (-1)^{j+1} \px^{2j+1}u =-K_{\lambda}u + F\]
with $u(0) = f$, equivalently,
\[u(t,x) = W(t)f - \int_0^t W(t-\tau) (K_{\lambda}u)(\tau) \; d\tau + \int_0^t W(t-\tau) F(\tau)\; d\tau.\]
Using \eqref{eq:K_lambda}, one has
\[\norm{u}_{Y_{I}^s} \lesssim \left(\norm{f}_{H^s} + \norm{\mathcal F^{-1}\left(\bra{\tau - k^{2j+1}}^{-1}\widetilde F\right)}_{Y_T^s}\right) + |I|^{1-\epsilon}\norm{u}_{Y_{I}^s},\]
which implies
\[\norm{u}_{Y_{I}^s} \le C(I)\left(\norm{f}_{H^s} + \norm{\mathcal F^{-1}\left(\bra{\tau - k^{2j+1}}^{-1}\widetilde F\right)}_{Y_T^s}\right),\]
for a proper non-empty $I \subset [0,T]$ with $|I| < 1$. Let $I = [0, t_0]$, then we divide $[0,T]$ into $\left[\frac{T}{t_0}\right]+1$ the subintervals, denoted by $I_j$, precisely, let $t_* := \left[\frac{T}{t_0}\right]$, set
\[I_0 = [0,t_0],  \quad I_j = [jt_0, (j+1)t_0], \quad j=1,2,\cdots, t_* -1, \quad \mbox{and} \quad I_{t_*} = [t_* t_0, T].\]
Therefore, we have\footnote{One may use time cut-off functions supported on each $I_j$ so that $u = \sum_{j=0}^{t_*} \eta_{I_j}(t)u$ on $[0,T]$. Then $Y_T^s$ norm of each part is bounded by $Y_{I_j}^s$ norm of the same one.}
\[\norm{u}_{Y_T^s} \lesssim (t_*+1)C(t_0)\left(\norm{f}_{H^s} + \norm{\mathcal F^{-1}\left(\bra{\tau - k^{2j+1}}^{-1}\widetilde F\right)}_{Y_T^s}\right),\]
which completes the proof.
\end{proof}

Using Lemmas \ref{lem:modified Xsb} and \ref{lem:Xsb} Item (4), one proves the local well-posedness of \eqref{eq:nonlinear closed loop}.
\begin{lemma}[Local well-posedness of nonlinear closed loop system]\label{lem:LWP of cl}
Let $s \ge 0$ and $T > 0$ be given. Let define a map $\Gamma : H^s \to C_TH^s$ as in the second part of the right-hand side of \eqref{eq:duhamel} (again denoted by $\Gamma u$). Then, there exists $\delta = \delta(T) > 0$ such that if
\[\norm{u_0}_{H^s} \le \delta,\]
then the map $\Gamma$ is a contraction map on a suitable ball. Moreover the map is locally uniformly continuous.
\end{lemma}
\begin{proof}
The proof is analogous to the proof of Lemma \ref{lem:CMP}. Taking $Y_T^s$ norm to the map $\Gamma u$ and applying Lemmas \ref{lem:modified Xsb} and \ref{lem:Xsb} Item (4), one has
\[\norm{\Gamma u}_{Y_T^s} \le C\norm{u_0}_{H^s} + C\norm{u}_{Y_T^s}^2\]
and for $u-\underline{u}$ with $u(0) =\underline{u}(0)$,
\[\norm{\Gamma u - \Gamma \underline{u}}_{Y_T^s} \le C\left(\norm{u}_{Y_T^s} + \norm{\underline{u}}_{Y_T^s}\right)\norm{u-\underline{u}}_{Y_T^s},\]
where the constant $C$ be the maximum one among constants appearing in Lemma \ref{lem:Xsb} (4) and Lemma \ref{lem:modified Xsb}. By taking $\delta >0$ satisfying $8C^2 \delta < 1$, we claim the map $\Gamma$ is contractive on a ball
\[\{v \in Y_T^s : \norm{v}_{Y_T^s} \le 2C\delta\}.\]
One similarly proves that the map is Lipschitz continuous, thus we complete the proof.
\end{proof}

The local solution constructed in Lemma \ref{lem:LWP of cl} can be extended to the global one.
\begin{theorem}[Global well-posedness]\label{thm:GWP} Let $s \ge0$ and $T > 0$ be given. For any $u_0 \in H^s$, there exists a unique solution $u$ to \eqref{eq:nonlinear closed loop} in  $Y_T^s \cap C_TH^{s}$ such that the following estimate holds true:
\[\norm{u}_{Y_T^s} \le \alpha_{T,s}(\norm{u_0}_{L^2}) \norm{u_0}_{H^s},\]
where $\alpha_{T,s}$ is positive nondecreasing continuous function depending on $T$ and $s$.
\end{theorem}
\begin{proof}
A direct computation, in addition to the fact that $G$ is self-adjoint in $L^2$, yields
\begin{equation}\label{eq:L^2 energy}
\frac12 \frac{d}{dt}\int_\T u^2  =(-1)^j\int_\T u\partial_x^{2j+1} u  - \int_\T u(uu_x) - \int_\T u K_{\lambda} u 
= -(GL_{\lambda}^{-1}u,Gu).
\end{equation}
Since $G$ and $L_{\lambda}^{-1}$ is bounded in $L^2$, the Cauchy-Schwarz and the Gr\"onwall's inequalities ensure
\begin{equation}\label{eq:L^2 global}
\norm{u(t)}_{L^2}^2 \lesssim \norm{u_0}_{L^2}^2e^{ct},
\end{equation}
for some $c>0$ depending on $G$ and $L_{\lambda}$, and $t >0$. Together with Lemma \ref{lem:LWP of cl} and the standard continuity argument, we complete the global well-posedness of \eqref{eq:nonlinear closed loop} in $L^2$.

Let $v = u_t$ for a smooth solution to \eqref{eq:nonlinear closed loop}. Then, $v$ solves
\begin{equation}\label{eq:2j+1}
\partial_t v + (-1)^{j+1} \partial_x^{2j+1} v + (uv)_x = -K_{\lambda}v
\end{equation}
where
\begin{equation}\label{eq:2j+1v_0}
v_0 = (-1)^j \partial_x^{2j+1} u_0 - u_0 \partial_x u_0 - K_{\lambda} u_0. 
\end{equation}
Note that $u$ is a solution to \eqref{eq:nonlinear closed loop}, thus we can take $T_0 > 0$ such that
\[\|u\|_{Y^0_{T_0}} \lesssim \|u_0\|_{L^2}.\]
Then, analogously as in the proof of Theorem \ref{lem:LWP of cl}, we also have
\[\|v\|_{L^{\infty}(0,T_1; L^2)} \lesssim \|v\|_{Y^0_{T_1}} \lesssim 2\|v_0\|_{L^2},\]
here $0 < T_1 < T_0$ is chosen appropriately. 

\medskip

On the other hand, a direct computation gives
\[\frac{d}{dt} \left(\frac12 \int_\T (\px^j u)^2 -  \frac16 \int_\T u^3\right) = -(GL_{\lambda}^{-1}\px^j u, G \px^j u) + \frac12 \int_\T u^2 K_{\lambda} u.\]
Using boundedness of $G$ and $L_{\lambda}^{-1}$, and thus $K_{\lambda}$, and Gagliardo–Nirenberg inequality, we have
\[\begin{aligned}
\|\px^j u(t)\|_{L^2}^2 \le&~{} \frac13 \int_T u^3 + \|\px^j u_0\|_{L^2}^2 + \frac13 \int_\T u_0^3 \\
&~{}+ C\int_0^t \|\px^j u(s)\|_{L^2}^2 \; ds + \int_0^t \frac14\|\px^j u(s)\|_{L^2}^2 + \frac14\| u(s)\|_{L^2}^6 + \frac12\|u(s)\|_{L^2}^2 \; ds\\
\le&~{} \frac{1}{12}\|\px^j u(t)\|_{L^2}^2 + \frac{1}{4}\|u(t)\|_{L^2}^{\frac{10}{3}} + \frac{13}{12}\|\px^j u_0\|_{L^2}^2 + \frac{1}{4}\|u_0\|_{L^2}^{\frac{10}{3}} \\
&~{}+ C\int_0^t \|\px^j u(s)\|_{L^2}^2 \; ds + \int_0^t \frac14\|\px^j u(s)\|_{L^2}^2 + \frac14\| u(s)\|_{L^2}^6 + \frac12\|u(s)\|_{L^2}^2 \; ds.
\end{aligned}\]
With \eqref{eq:L^2 global}, we claims from Gr\"onwall's inequality that $\|\px^j u(t)\|_{L^2}$ does not blow up in finite time, thus so $\|\px u(t)\|_{L^{\infty}}$.

\medskip

A similar computation as in \eqref{eq:L^2 energy} yields
\[\frac12\frac{d}{dt}\int_\T v^2  = -\frac12 \int_\T u_xv^2  -(GL_{\lambda}^{-1}v,Gv),\]
which ensures that
\[\|v(t)\|_{L^2} \le e^{Ct +\int_0^t \|\px u(s)\|_{L^{\infty}}}\|v_0\|_{L^2}\]
for all $t > 0$, thanks to global boundedness of $\|\px u(u)\|_{L^{\infty}}$. Finally, for given $T > 0$, a direct computation with $v = u_t = - (-1)^{j+1} \partial_x^{2j+1} u -  uu_x -K_{\lambda}u$
gives
\[\begin{aligned}
\|\partial_x^{2j+1}u\|_{L^2} \lesssim&~{} \|u\|_{L^2} + \|v\|_{L^2} + \|u\|_{L^2}\|u_x\|_{L^{\infty}}\\
\le&~{} C(\|u_0\|_{L^2} + \|v_0\|_{L^2} + \|u_0\|_{L^2}^3) + \frac12 \|\partial_x^{2j+1}u\|_{L^2},
\end{aligned}\]
for some $C>0$ depending on $T$, which in addition to \eqref{eq:2j+1v_0} implies
\[\|u\|_{L^{\infty}(0,T;H^{2j+1})} \le \alpha_{T,2j+1}(\norm{u_0}_{L^2})\|u_0\|_{H^{2j+1}}.\]
For $s \in (2j+1) \N$, one can show the global well-posedness similarly, and for $(2j+1)(n-1) < s < (2j+1)n$, $n \in \N$, it follows from the interpolation argument. Therefore, we complete the proof.
\end{proof}

\section{Local controllability and stability: Nonlinear results} \label{sec5} This section devotes to proving  Theorems \ref{control_zhang} and \ref{stability_zhang}.
\subsection{Proof of Theorem \ref{control_zhang}}
Rewrite the system \eqref{k5} in its equivalent integral equation form:
\begin{equation}\label{eq:lr}
u(t)=W(t) u_{0}+\int_{0}^{t} W(t-\tau)(G h)(\tau) d \tau-\int_{0}^{t} W(t-\tau)\left(u u_{x}\right)(\tau) d \tau.
\end{equation}
Define
$$\omega(T, u) :=\int_{0}^{T} W(T-\tau)\left(u u_{x}\right)(\tau) d \tau.$$
Then, Lemma \ref{lem:Xsb} (3) and (4) yield
\begin{equation}\label{eq:wT}
\begin{aligned}
\norm{\omega(T,u)}_{H^s} \lesssim&~{} \left\|\bra{k}^s\mathcal F\left(\int_{0}^{t} W(t-\tau)\left(u u_{x}\right)(\tau) d \tau\right)\right\|_{\ell_k^2L_{\tau}^1}
\lesssim~{} \norm{u}_{Y_T^s}^2 < \infty,
\end{aligned}
\end{equation}
provided that $u \in Y_T^s$. Choose $h=\Phi(u_0,u_1+\omega(T,u))$ in the equation \eqref{eq:lr} for $u \in Y_T^s$. From Theorem \ref{thm:EC}, we have that for given $u_0$ and $u_1$
\begin{equation}\label{eq:map}
u(t)=W(t) u_{0}+\int_{0}^{t} W(t-\tau)\left(G \Phi\left(u_{0}, u_{1}+\omega(T, u)\right)\right)(\tau) d \tau-\int_{0}^{t} W(t-\tau)\left(u u_{x}\right)(\tau) d \tau
\end{equation}
with
$\left.u\right|_{t=0}=u_{0}\quad \mbox{and} \quad \left.u\right|_{t=T}=u_{1}.$
Then, the following lemma proves Theorem \ref{control_zhang}.
\begin{lemma}\label{lem:CMP}
Let $s \ge 0$ and $T > 0$ be given. Let define a map $\Gamma : H^s \to C_TH^s$ as in \eqref{eq:map} (denoted by $\Gamma u$). Then, there exists $\delta = \delta(T) > 0$ such that if $\norm{u_0}_{H^s}\le \delta$ and $\norm{u_1}_{H^s} \le \delta$,
then the map $\Gamma$ is a contraction map on a suitable ball.
\end{lemma}
\begin{remark}
The standard Picard iteration argument ensures the uniqueness of the fixed point, hence the condition $u(T,x) = u_1$ is guaranteed.
\end{remark}

\begin{proof}[Proof of Lemma \ref{lem:CMP}]
We denote the maximum implicit constant among ones appearing in Lemma \ref{lem:Xsb}, \eqref{eq:phi} and \eqref{eq:wT} by $C >0$. Note that here the constant $C$ depends on time $T>0$, precisely, $C$ is increasing when $T$ grows up.

So, using Lemma \ref{lem:Xsb} with Remark \ref{rem:nonlinear embedding}, \eqref{eq:phi} and \eqref{eq:wT}, one has
\[\norm{\Gamma u}_{Y_T^s} \le C\left( \norm{u_0}_{H^s} + \norm{u}_{Y_T^s}^2 + \left(\norm{u_0}_{H^s} + \norm{u_1}_{H^s} + \norm{u}_{Y_T^s}^2\right)\right).\]
Analogously, for solutions $u$ and $\underline{u}$ with $u(0) = \underline{u}(0)$ and $u(T) = \underline{u}(T)$, we have
\[\norm{\Gamma u - \Gamma \underline{u}}_{Y_T^s} \le 2C\left(\norm{u}_{Y_T^s} + \norm{\underline{u}}_{Y_T^s}\right)\norm{u-\underline{u}}_{Y_T^s}.\]
Taking $\delta >0$ satisfying $48C^2\delta \le 1$\footnote{Here $\delta$ depends on $T$ since the constant relies on $T$.}, we conclude that the map $\Gamma$ is contractive in a ball
\[\{v \in Y_T^s : \norm{v}_{Y_T^s} \le 6C\delta\},\]
thus this completes the proof.
\end{proof}

\subsection{Proof of Theorem \ref{stability_zhang}}We are now ready to prove Theorem \ref{stability_zhang}.
For given $s \ge 0$ and $\lambda >0$, we have from Proposition \ref{zhang_2} that
\[\norm{W_{\lambda}(t)u_0}_{H^s} \le Ce^{-\lambda t} \norm{u_0}_{H^s},\]
where the implicit constant $C>0$ depends only  on $s$. For any $0 < \lambda' < \lambda$, take $T = T(\lambda') > 0$ such that
\[2Ce^{-\lambda T} \le e^{-\lambda' T}.\]

Let us consider solution $u$ to the integral equation \eqref{eq:duhamel} as a fixed point of the map
$$
\Gamma u(t)=W_{\lambda}(t) u_{0}-\int_{0}^{t} W_{\lambda}(t-\tau)\left(u u_{x}\right)(\tau) d \tau
$$
in some closed ball $B_{R}(0)$ in the function space $Y_T^s$. This will be done provided that $\left\|u_{0}\right\|_{H^s} \leq \delta$ where $\delta$ is a small number to be determined. Furthermore, to ensure the exponential stability with the claimed decay rate, the numbers $\delta$ and $R$ will be chosen in such a way that
$$
\|u(T)\|_{H^s} \leq e^{-\lambda^{\prime} T}\left\|u_{0}\right\|_{H^s}.
$$
Applying Lemmas \ref{lem:modified Xsb} and \ref{lem:Xsb} (4), there exist some positive constant $C_{1}, C_{2}$ (independent of $\delta$ and $R$ ) such that
\[\norm{\Gamma u}_{Y_T^s} \le C_1\norm{u_0}_{H^s} + C_2\norm{u}_{Y_T^s}^2\]
and for $u-\underline{u}$ with $u(0) =\underline{u}(0)$,
\[\norm{\Gamma u - \Gamma \underline{u}}_{Y_T^s} \le C_2\left(\norm{u}_{Y_T^s} + \norm{\underline{u}}_{Y_T^s}\right)\norm{u-\underline{u}}_{Y_T^s}.\]
On the other hand, we have for some constant $C^{\prime}>0$ and all $u \in B_{R}(0)$
$$
\begin{aligned}
\|\Gamma(u)(T)\|_{Y_T^s} & \leq C_{1}\left\|W_{\lambda}(T) u_{0}\right\|_{Y_T^s}+C_{2}\left\|\int_{0}^{T} W_{\lambda}(T-\tau)\left(u u_{x}\right)(\tau) d \tau\right\|_{Y_T^s} \\
& \leq e^{-\lambda T} \delta+C^{\prime} R^{2}.
\end{aligned}
$$

Pick $\delta=C_{4} R^{2},$ where $C_{4}$ and $R$ are chosen so that
$$
\frac{C^{\prime}}{C_{4}} \leq C e^{-\lambda T},\quad \left(C_{1} C_{4}+C_{2}\right) R^{2} \leq R \quad \text{and} \quad 2 C_{2} R \leq \frac{1}{2}.
$$
Then we have
$$
\|\Gamma(u)\|_{Y_T^s} \leq R, \quad\forall u \in B_{R}(0)
$$
and
$$
\left\|\Gamma\left(u_{1}\right)-\Gamma\left(u_{2}\right)\right\|_{Y_T^s} \leq \frac{1}{2}\left\|u_{1}-u_{2}\right\|_{\mathbb{Z}_{T, \sigma}^{T}}, \quad\forall u_{1}, u_{2} \in B_{R}(0).
$$
Therefore, $\Gamma$ is a contraction in $B_{R}(0) .$ Furthermore, its unique fixed point $u \in B_{R}(0)$ fulfills
$$
\|u(T)\|_{H^s} \leq\|\Gamma(u)(T)\|_{Y_T^s} \leq e^{-\lambda^{\prime} T} \delta
$$

Assume now that $0<\left\|u_{0}\right\|_{0}<\delta .$ Changing $\delta$ into $\delta^{\prime} \equiv\left\|u_{0}\right\|_{s}$ and $R$ into $R^{\prime} \equiv$ $\left(\delta^{\prime} / \delta\right)^{\frac{1}{2}} R,$ we infer that
$$
\|u(T)\|_{H^s} \leq e^{-\lambda^{\prime} T}\left\|u_{0}\right\|_{H^s}
$$
and an obvious induction yields
$$
\|u(n T)\|_{H^s} \leq e^{-\lambda^{\prime} n T}\left\|u_{0}\right\|_{H^s}
$$
for any $n \geq 0 .$ We infer by the semigroup property that there exists some positive constant $C>0$ such that
$$
\|u(t)\|_{H^s} \leq C e^{-\lambda{t}}\left\|u_{0}\right\|_{H^s},
$$
if $\left\|u_{0}\right\|_{H^s} \leq \delta .$ The proof is complete.\qed

\subsection*{Acknowledgments}
R. de A. Capistrano--Filho was supported by CNPq 408181/2018-4, CAPES-PRINT 88881.311964/2018-01, CAPES-MATHAMSUD 88881.520205/2020-01, MATHAMSUD 21-MATH-03 and Propesqi (UFPE). C. Kwak was partially supported by the Ewha Womans University Research Grant of 2020 and the National Research Foundation of Korea (NRF) grant funded by the Korea government (MSIT) (No. 2020R1F1A1A0106876811). F. Vielma Leal was partially supported by FAPESP/Brazil grant 2020/14226-4.  This work was carried out while the second author was visiting the Federal University of Pernambuco. He thanks the institution for their hospitality.

\end{document}